\documentclass{amsproc}
\usepackage{euscript}
\usepackage{cases}
\usepackage{mathrsfs}
\usepackage{bbm}
\usepackage{amssymb}
\usepackage{txfonts}
\usepackage{amsfonts,amsmath,amsxtra,mathdots,mathabx}
\usepackage{color}
\usepackage{hyperref}
\usepackage{tikz}

\allowdisplaybreaks

\DeclareFontFamily{U}{matha}{\hyphenchar\font45}
\DeclareFontShape{U}{matha}{m}{n}{
	<5> <6> <7> <8> <9> <10> gen * matha
	<10.95> matha10 <12> <14.4> <17.28> <20.74> <24.88> matha12
}{}
\DeclareSymbolFont{matha}{U}{matha}{m}{n}

\DeclareMathSymbol{\Lt}{3}{matha}{"CE}
\DeclareMathSymbol{\Gt}{3}{matha}{"CF}

\DeclareSymbolFont{mathc}{OML}{txmi}{m}{it}% txfonts
\DeclareMathSymbol{\varvv}{\mathord}{mathc}{118} 
\DeclareMathSymbol{\varww}{\mathord}{mathc}{119} 
\DeclareMathSymbol{\varnu}{\mathord}{mathc}{"17}

\DeclareSymbolFont{mathd}{OML}{ztmcm}{m}{it}
\DeclareMathSymbol{\varalpha}{\mathord}{mathd}{11}

\DeclareMathSymbol{\vvepsilon}{\mathord}{mathd}{15}
\def\vepsilon{\scalebox{0.87}{$\vvepsilon$}}
\def\sepsilon{\scalebox{0.62}{$\vvepsilon$}}
\def\ssepsilon{\scalebox{0.5}{$\vvepsilon$}}

\newcommand{\BC}{{\mathbb {C}}} 
 
 \newcommand{\BH}{{\mathbb {H}}}

\newcommand{\BQ}{{\mathbb {Q}}} \newcommand{\BR}{{\mathbb {R}}}

 \newcommand{\BZ}{{\mathbb {Z}}}

 \newcommand{\RN}{{\mathrm {N}}}

\def\frO{\text{$\text{\usefont{U}{BOONDOX-cal}{m}{n}O}$}\hskip 1pt}
\def\frOO{\text{$\text{\usefont{U}{BOONDOX-cal}{m}{n}O}$} }

\newcommand{\GL}{{\mathrm {GL}}} \newcommand{\PGL}{{\mathrm {PGL}}}
\newcommand{\SL}{{\mathrm {SL}}} \newcommand{\PSL}{{\mathrm {PSL}}}

 \newcommand{\Tr}{{\mathrm{Tr}}}

\newcommand{\ds}{\displaystyle}
\newcommand{\sstyle}{\scriptstyle}

\newcommand{\ra}{\rightarrow}

\def\fra{\mathfrak{a}}
\def\-{^{-1}}

\def\mod{\mathrm{mod}\ }

\def\sumx{\sideset{}{^\star}\sum}

\def\tw{\textit{w}}

\def\lp {\left (}
\def\rp {\right )}

\def\boldJ {\boldsymbol J}

\renewcommand{\Im}{{\mathrm{Im} }}
\renewcommand{\Re}{{\mathrm{Re} }}

\def\tw{\textit{w}} 
\def\shskip{\hskip 0.5 pt}

\def\frd{\mathfrak{d}}
\def\frm{\mathfrak{m}}
\def\frn{\mathfrak{n}}
\def\fp{\mathfrak{p}}
\def\fq{\mathfrak{q}}
\def\frv{\mathfrak{v}}
\def\Avg{\mathrm{Avg}}

\def\ssstyle{\scriptscriptstyle}

\def\SSB{\text{\mbox{\larger[1]$\text{\usefont{U}{dutchcal}{m}{n}B}$}}}
\def\SB{\raisebox{- 2 \depth}{$\SSB$}}
\def\SSD{\text{\mbox{\larger[1]$\text{\usefont{U}{dutchcal}{m}{n}D}$}}}
\def\SD{\raisebox{- 2 \depth}{$\SSD$}}
\def\SSP{\text{\mbox{\larger[1]$\text{\usefont{U}{dutchcal}{m}{n}P}$}}}
\def\SP{\raisebox{- 2 \depth}{$\SSP$}}
\def\SSE{\text{\mbox{\larger[1]$\text{\usefont{U}{dutchcal}{m}{n}E}$}}}
\def\SE{\raisebox{- 2 \depth}{$\SSE$}}

\makeatletter
\g@addto@macro\normalsize{\setlength\abovedisplayskip{3pt}}
\makeatother

\makeatletter
\g@addto@macro\normalsize{\setlength\belowdisplayskip{3pt}}
\makeatother

\newcommand{\delete}[1]{}

\theoremstyle{plain}

\newtheorem{thm}{Theorem}[section] \newtheorem{cor}[thm]{Corollary}
\newtheorem{lem}[thm]{Lemma}  \newtheorem{prop}[thm]{Proposition}

 \newtheorem{defn}[thm]{Definition}

\newtheorem {rem}[thm]{Remark}

\newtheorem*{acknowledgement}{Acknowledgements}

\numberwithin{equation}{section}

\begin{document}

	\title[Low-lying zeros over imaginary quadratic fields]{Low-lying zeros of $L$-functions for Maass forms over imaginary quadratic fields}

\author{Sheng-Chi Liu}%
\address{Department of Mathematics and Statistics, Washington State University,
	Pullman, WA 99164-3113, USA}%
\email{scliu@math.wsu.edu}%

\author{Zhi Qi}
\address{School of Mathematical Sciences\\ Zhejiang University\\Hangzhou, 310027\\China}
\email{zhi.qi@zju.edu.cn}


\thanks{The first author was supported by a grant (\#344139) from the Simons Foundation.}


\begin{abstract}
	We study the $1$- or $2$-level density of families of $L$-functions for Hecke--Maass forms over an imaginary quadratic field $F$. For test functions whose Fourier transform is supported in $\left(-\frac 32, \frac 32\right)$, we prove that the $1$-level density for Hecke--Maass forms over $F$ of square-free level $\mathfrak{q}$, as $\mathrm{N}(\mathfrak{q})$ tends to infinity, agrees with that of the orthogonal random matrix ensembles. For Hecke--Maass forms over $F$ of full level, we prove similar statements for the $1$- and $2$-level densities, as the Laplace eigenvalues tend to infinity.
	\end{abstract}

	\maketitle

\section{Introduction}

Katz and Sarnak \cite{KS-1,KS-2} predict that the
limiting behavior of  the low-lying zeros (zeros near the central point $s = \frac 1 2$) of  a family
of $L$-functions 
agrees with that of the eigenvalues near $1$ of  the ensemble of random matrices associated to the family. 
There is now a vast literature on verifying this conjectural agreement (the Katz--Sarnak heuristics) for various families of $L$-functions.  We refer the reader to  the pioneer work of Iwaniec--Luo--Sarnak \cite{ILS-LLZ}, and also the survey articles \cite{Miller-Survey1,Miller-Survey2} for an extensive set of references.

The purpose of this paper is to verify the Katz--Sarnak heuristics (for suitably restricted test
functions) for families of Hecke--Maass forms over an imaginary quadratic field. Our work is in parallel with \cite{Miller-Maass,Alpoge-Miller-1}, where they consider Maass forms over $\BQ$. Moreover, we shall give a  simple proof of (a variant of) the result in \cite{Alpoge-Miller-1} and, in addition, consider the families of even and odd Maass forms for  $\SL_2 (\BZ)$. 

\subsection{The Katz--Sarnak density conjecture}

%Assume the Generalized Riemann Hypothesis (GRH) for the family of $L$-functions so that their non-trivial zeros lie on the critical line $\Re (s) = \frac 1 2$.  
The main statistic we study in this paper is the $1$- or $2$-level  density.

\begin{defn}[1- and 2-level densities]\label{def: 1-level, 2-level density} Define the $1$-level density of an $L$-function $ L(s, f)$ with non-trivial zeros $\frac 1 2 +i \gamma_f$  {\rm(}note that $\gamma_f \in \BR$ if the generalized Riemann hypothesis holds for $L (s, f)${\rm)} by
	\begin{align}\label{0eq: 1-level density}
	D_1 (f, \phi, R) = \sum_{ \gamma_f} \phi \bigg(  \frac {\log R} {2\pi} \gamma_f \hskip -1pt \bigg),
	\end{align}
	where $\phi : \BR \ra \BC$ is an even Schwartz function such that its Fourier transform 
	\begin{align}\label{0eq: Fourier}
		\widehat{\phi} (y) = \int_{-\infty}^{\infty} \phi (x) e (-xy) d x, 
	\end{align}
	has compact support {\rm(}as usual, $e (x) = e^{2\pi i x}${\rm)}, and $R$ is a scaling parameter. Similarly, the 2-level density is defined by
	\begin{align}\label{0eq: 2-level density}
	D_2 (f, \phi_1, \phi_2, R) = \mathop {\sum \sum}_{ j_1 \neq \pm j_2 } \phi_1 \bigg(  \frac {\log R} {2\pi} \gamma_{f}^{(j_1)} \hskip -1pt \bigg) \phi_2 \bigg(  \frac {\log R } {2\pi} \gamma_{f}^{(j_2)} \hskip -1pt \bigg), 
	\end{align}
	where the zeros $\frac 1 2 + i \gamma_f^{(j)}$ are labeled such that $\gamma_f^{(-j)} = - \gamma_f^{(j)}$. The $n$-level density may be defined in the same way. 
\end{defn}

\begin{rem}\label{rem: GRH}
	The definitions in {\rm\eqref{0eq: 1-level density}} and {\rm\eqref{0eq: 2-level density}} make sense {\rm(}even if $\gamma_f$ are not real{\rm)} irrespective of the Riemann hypothesis for $L(s, f)$ because  $\phi$ is entire. Unless otherwise stated, we shall not assume any generalized Riemann hypothesis in this paper {\rm(}the Riemann hypothesis for the Riemann $\zeta$-function or Dirichlet $L$-functions will be assumed only in the two theorems in Appendix \ref{sec: app real}{\rm)}. 
\end{rem}

To be precise, let $\mathscr{F} = \bigcup \mathscr{F} (Q)$ be  a  family of automorphic forms, ordered by their conductors $Q$. The Katz--Sarnak density conjecture states that as the conductor $Q \ra \infty$,  the averaged $n$-level density %\footnote{As this moment, we do not want to make the definition precise, as we shall need to also average on $Q = t_f$   by a weight function that localizes $t_f$ to a neighborhood of $T$ and then let $T \ra \infty$.} %(weighted by certain non-negative $w$) of the $L$-functions for  $\mathscr{F} (Q)$, defined by $
$\SD_n(\mathscr{F} (Q),\phi ) $ (its definition will be made precise in our setting as in (\ref{0eq: 1-level density, level}, \ref{0eq: 1-level density, t}, \ref{0eq: 2-level density, t}) with extra weights and averages) %= \frac{\sum_{f \in \mathscr{F} (Q)} w (f) D_n (f, \phi, Q )  }{\sum_{f \in \mathscr{F} (Q)}   w (f) } , $$ 
converges to   the limiting $n$-level density for the associated symmetry group $\mathrm{G} (\mathscr{F})$:
\begin{equation}
\lim_{Q\to\infty} \SD_n(\mathscr{F} (Q),\phi ) = \int \cdots\int  \phi(x_1,\dots,x_n)W_{n }(\mathrm{G}(\mathscr{F}))(x_1,\dots,x_n)dx_1\cdots dx_n,\end{equation}
where $\phi (x_1, \dots, x_n) = \phi_1 (x_1) \dots \phi_n (x_n)$, 
\begin{center}
	\begin{tabular}{cc} \\
		$\underline{\ \ \ \ \ \ \ \ \ \ \ \ \mathrm{G} \ \ \ \ \ \ \ \ \
			\ \ }$ & $\underline{\ \ \ \ \ \ \ \ \ \ \ \ \raisebox{0.9 \depth}{$W_{n} (\mathrm{G})$} \
			\ \ \ \ \ \ \ \ \ \ \ }$ \\ \\
		${\rm U }$ & $\det\left(K_0(x_j,x_k)\right)_{ 
			j, \shskip k
			\leqslant n}$ \\ \\
		${\rm Sp }$ & $\det\left(K_{-1}(x_j,x_k)\right)_{ j, \shskip k \shskip \leq \shskip
			n}$ \\ \\
		${\rm SO(even)}$ & $\det\left(K_{1}(x_j,x_k)\right)_{  j, \shskip k \shskip \leqslant \shskip n}$ \\ \\
		${\rm SO(odd)}$ & $\det\left(K_{-1}(x_j,x_k)\right)_{ j, \shskip k \shskip \leq \shskip
			n}  $ \\ \\ 
		$ $ & $ \hskip 40 pt +   \sum_{\varnu=1}^n  \delta_0 (x_{\varnu}) \det\left(K_{-1}(x_j,x_k)\right)_{ j, \shskip k \shskip \neq \varnu}$
	\end{tabular}
\end{center}
and $$W_n ({\rm O}) = \frac 1 2 \big( W_n ({\rm SO(even)}) + W_n ({\rm SO(odd)}) \big),$$  
with
\begin{align*}
K_\epsilon(x,y) = \frac{\sin\left(\pi(x-y) \right)}{\pi(x-y)} +
\epsilon \frac{\sin\left(\pi(x+y) \right)}{\pi(x+y)} .
\end{align*} 

It is often convenient to look at the Fourier transform side; for the $1$-level densities, we have
\begin{align*}
 \widehat{W }_{1} ({\rm O})(y) & = \delta_0(y) + \frac 12,
\\ 
  \widehat{W}_{1}  ({\rm SO(even)})(y) & = \delta_0(y) + \frac12 \eta(y),
 \\ 
  \widehat{W}_{1}   ({\rm SO(odd)})(y) & = \delta_0(y) - \frac12 
\eta(y) + 1 , \\ 
\widehat{W}_{1}  ({\rm U})(y) & = \delta_0(y), \\
 \widehat{W}_{1}  ({\rm Sp})(y) & = \delta_0(y) -
\frac12 \eta(y) , 
\end{align*}
where $\eta(y) = 1$, $\frac 1 2$, $0$ for $|y|<1$,
$y = \pm 1$, $|y| > 1$, and $\delta_0 (y)$ is the Dirac
delta distribution at $y = 0$.  

The different classical compact
groups  have distinguishable $1$-level densities when the support of $\widehat {\phi}$ exceeds $[-1, 1]$. However, for support in $(−1, 1)$ the
three orthogonal flavors are mutually indistinguishable (though they are
different from unitary and symplectic). 
In cases that it is infeasible to extend the support beyond $[-1, 1]$, one could study the $2$-level density, which Miller \cite{Miller-Thesis,Miller-2-level} showed distinguishes the three orthogonal ensembles
for arbitrarily small support.

\subsection*{Setup}

Let $F $  be an imaginary quadratic field, and let $\frO$ be its ring of integers.  Assume for simplicity that the class number   $h_F = 1$.

Let $\Gamma_0 (\fq) $ be the Hecke congruence subgroup of $ \GL_2 (\frO)$\footnote{In the literature, many authors like to consider congruence subgroups of $ \SL_2 (\frO)$. The translation between $\SL_2$ and $\GL_2$ however is usually straightforward (see for example \cite{Venkatesh-BeyondEndoscopy}). %The cusp forms considered here are conventionally called {\it even} in the $\PSL_2$ setting. 
	The main reason for working on $\GL_2 (\frO)$ is to make the definitions of Hecke operators and $L$-functions valid over ideals.} with {\it square-free} level $\fq$. Let $ H^{\star} (\fq)  $ be the collection of primitive Hecke--Maass newforms of level $\fq$. For each $f \in H^{\star} (\fq)  $, we denote by $    1 + 4 t_f^2$ the Laplacian eigenvalue of $f $, and by $\omega^{\star}_f$ the spectral weight of $f $ (as defined in \eqref{2eq: omega*}).  We use the notation $\mathrm{Avg}_{\fq}(A; w)$ to denote the average of $A$ over $ H^{\star} (\fq) $ weighted by $w $. That is,
\begin{align*}
	\Avg_{\fq} (A; w) = \frac{\sum_{f \in H^{\star} (\fq)} w (f) A(f )  }{\sum_{f \in H^{\star} (\fq)}   w (f) } .
\end{align*}

\subsection{Main results}

We start with describing the space of weight functions. 
The space $ \mathscr{H} (S, N) $ below is introduced in \cite[Definition 5.1]{BM-Kuz-Spherical} and \cite[Definition 4]{Venkatesh-BeyondEndoscopy}.

\begin{defn}[Space of weight functions]\label{def: test functions}
	Let $ S > \frac 3 2 $ and $N > 6$. We set  $ \mathscr{H} (S, N) $ to be the space of functions $h : \BR \ra \BC$ which extend to an even holomorphic function on the strip $\{ t + i \sigma : |\sigma | \leqslant S \}$ such that, for   $|\sigma| < S $, we have uniformly
	\begin{align}\label{0eq: bounds for h(t)}
		h (t+ i \sigma)  \Lt e^{- \pi |t|} (|t|+1)^{-N}. 
	\end{align}
	Let $ \mathscr{H}^{\ssstyle +} (S, N) $ be the space of non-negative valued functions $h (t) \nequiv 0$ in  $ \mathscr{H}  (S, N) $. 
	
	For $1 < M < T$ and $h (t) \in \mathscr{H}^{\ssstyle +}   (S, N)$, define 
	\begin{align}\label{0eq: defn of hTM}
		h_{T, \shskip  M} (t) = h ((t-T)/ M) + h ((t+T)/ M),
	\end{align}
	which also lies in the space $ \mathscr{H}^{\ssstyle +} (S, N) $. 
\end{defn}

%We highly recommend the reader to jump to Appendix \ref{sec: app real} at this point. After comparing the analysis therein and that in \cite{Alpoge-Miller-1}, the reader should be convinced that the $h_{T, \shskip  M} (t)$ are indeed the rightful choice of weight functions.  

\vskip 5pt

\subsubsection{Low-lying zeros in the level aspect} 

For $h (t) \in \mathscr{H}^+ (S, N) $ as in Definition \ref{def: test functions}, we let $R = \RN (\fq)$ and define 
\begin{align}\label{0eq: 1-level density, level}
\SD_1 (H^{\star} (\fq), \phi; h) = \Avg_{\fq} \big( D_1 (f , \phi, \RN (\fq)); \omega^{\star}_f h (t_f) \big). 
\end{align}

\begin{thm}\label{thm: level-aspect}
	Fix $h   \in \mathscr{H}^+ (S, N)$. Let  $\phi $ be an even Schwartz function  with the support of $\widehat{\phi}$ in $\left(-\frac 32, \frac 3 2\right)$. Let  $\fq$ be  square-free.  Then 
	\begin{align*}
\lim_{\RN (\fq) \ra \infty} 	 \SD_1 (H^{\star} (\fq), \phi; h) = \int_{-\infty}^{\infty} \phi (x) W_1 (\mathrm{O}) (x) d x. 
	\end{align*}
\end{thm}

\begin{rem}
	It is possible to remove
	the spectral weights $\omega_f^{\star}$ in the definition {\rm\eqref{0eq: 1-level density, level}}, and this is done in {\rm\cite{ILS-LLZ}}. For this, one needs to assume the   Riemann hypothesis for $L(s, f)$ and $L(s, \mathrm{Sym^2} f)$. 
\end{rem}

\vskip 5pt

\subsubsection{Low-lying zeros in the $t_f$-aspect} 

Next we investigate the case where $\fq$ is fixed. For ease of exposition we take $\fq = (1)$.   
For $h_{T, \shskip  M} (t)$ as in \eqref{0eq: defn of hTM}, we let $R = T^4$ and define
\begin{align}\label{0eq: 1-level density, t}
\SD_1 (H^{\star} (1), \phi; h_{T, \shskip  M}) = \Avg_{(1)} \big( D_1 (f , \phi, T^4); \omega^{\star}_f h_{T, \shskip  M} (t_f) \big). 
\end{align}

A typical choice of the weight function used by Xiaoqing Li in  \cite{XLi2011} for the subconvexity problem in the $t_f$-aspect is the following $$	h_{T, \shskip  M} (t) = e^{- (t  - T)^2 / M^2} + e^{-(t + T)^2 / M^2}.$$
The class of weight functions $h_{T, \shskip  M} (t)$  in Definition \ref{def: test functions} serves very well the purpose of localizing to the conductors near $T$ and works far better than the class of $h_T (t)$ used by Alpoge and Miller \cite{Alpoge-Miller-1}. 

In the real case, by the arguments in  \cite{XLi2011}, one may easily deduce that  the corresponding (real)  Bessel integral $H_{T, \shskip  M}^{\ssstyle +} (x)$ is negligibly small for any $1 < x \Lt T M^{1-\sepsilon}$ as long as $M = T^{\mu}$ for some $0 < \mu < 1$. We may therefore avoid the complicated analysis in \cite{Alpoge-Miller-1} which occupies 9 pages. See Appendix \ref{sec: app real} for a simple proof of their result (on the $1$-level density for the family of $\SL_2$ Hecke--Maass forms over $\BQ$). 

When  one distinguishes the even and odd Maass forms, there also arises the Bessel integral $ H_{T, \shskip  M}^{\ssstyle -} (x) $. However, it behaves quite differently---$ H_{T, \shskip  M}^{\ssstyle -} (x) $ has an essential support on a neighborhood of $T/2\pi$ with length $M^{1+\sepsilon}$. Under the Riemann hypothesis for classical Dirichlet  $L$-functions, it turns out that $ H_{T, \shskip  M}^{\ssstyle -} (x) $ would contribute a main term to the asymptotics. See Appendix \ref{sec: even and odd}. 

In the complex case, however, the analogous analysis can only show that the (complex) Bessel integral $H_{T, \shskip  M} (x + iy)$ is negligibly small for $1 < |x| \Lt T M^{1-\sepsilon}$ but $|y| \Lt T$. Indeed, there is a dramatic change of behavior of $H_{T, \shskip  M} (x + iy)$ in the transition range $ |y| \asymp T $. When $  |y| \Gt T$, $H_{T, \shskip  M} (x + iy)$  is essentially supported on a {\it very narrow} sector-like region enclosing the $y$-axis. As a consequence, the angular (or horizontal) derivative $ \partial H_{T, \shskip  M} (x e^{i \theta}) / \partial \theta  $ could be {\it very large} on this supportive region.  See Lemma \ref{lem: HTM(z), 2} and Remark \ref{rem: HTM}.

\begin{thm}\label{thm: t-aspect, 1}
	Let $T, M > 1$ be such that $M = T^{\mu}$ with $0 < \mu < 1$. Fix $h   \in \mathscr{H}^+ (S, N)$ and define $h_{T, \shskip  M} $ by {\rm\eqref{0eq: defn of hTM}} in Definition {\rm\ref{def: test functions}}. Let  $\phi $ be an even Schwartz function  with the support of $\widehat{\phi}$ in $\left(-1, 1\right)$. Then 
	\begin{align*}
	\lim_{T \ra \infty} 	\SD_1 (H^{\star} (1), \phi ; h_{T, \shskip  M}) = \int_{-\infty}^{\infty} \phi (x) W_1 (\mathrm{SO(even)}) (x) d x. 
	\end{align*}
\end{thm}

Even if we assume the Riemann hypothesis for Hecke-character $L$-functions over $F$, it is unlikely that  the support of $\widehat{\phi}$ can be extended beyond the segment $[-1, 1]$. This will be explained in Remark   \ref{rem: failure [-1, 1]}. Suffice it to say, the arguments that worked in the real case are invalidated by the largeness of  $ \partial H_{T, \shskip  M} (x e^{i \theta}) / \partial \theta$. %\footnote{This is another instance that the angular variable makes it impossible to generalize the analysis from $\BR$ to $\BC$ in a parallel way. Similar obstacles arise when one attempts to extend the subconvexity result in \cite{XLi2011} from $\BQ$ to $\BQ (\sqrt{- D})$. Fortunately, thanks to the work of Miller \cite{Miller-Thesis,Miller-2-level}, we have the $2$-level density as an alternative here.}
Therefore we must resort to the averaged $2$-level density
\begin{align}\label{0eq: 2-level density, t}
\SD_2 (H^{\star} (1), \phi_1, \phi_2; h_{T, \shskip  M}) = \Avg_{(1)} \big( D_2 (f , \phi_1, \phi_2, T^4); \omega^{\star}_f h_{T, \shskip  M} (t_f) \big).
\end{align}

According to \cite[Theorem 3.2]{Miller-2-level}, for an even  function $\widehat\phi_1 (y_1) \widehat\phi_2 (y_2)$ supported in $|y_1| + |y_2| < 1$,
\begin{equation}
\begin{split}
\scalebox{0.99}{\text{$\ds	\int \hskip -4pt \int \hskip -1pt \widehat{\phi}_1(y_1) \widehat{\phi}_2(y_2) \widehat{W}_{2} (\hskip -0.5pt\mathrm{SO(even)} \hskip -0.5pt) (y_1, \hskip -1pt y_2) d y_1 d y_2 \hskip -0.5pt = \hskip -1.5pt \lp \hskip -0.5pt \widehat{\phi}_1(0) \hskip -1pt + \hskip -1pt \tfrac{1}{2} \phi_1(0) \hskip -0.5pt \rp \hskip -1pt \lp \hskip -0.5pt \widehat{\phi}_2 (0) \hskip -1pt + \hskip -1pt \tfrac{1}{2} \phi_2(0) \hskip -0.5pt \rp $}} &	\\
\scalebox{0.99}{\text{$\ds + 2 	\int_{-\infty}^{\infty}  |y|  \widehat{\phi_1}(y)  \widehat{\phi}_2 (y) d y 
- 2  \widehat{\phi_1 \phi}_2 (0) -  \phi_1(0) \phi_2(0)$}}, &
\end{split}
\end{equation}
and for arbitrarily small support, the three $2$-level densities for the groups $\mathrm{O}$, $\mathrm{SO(even)}$ and $\mathrm{SO(odd)}$ differ. Thus the following theorem implies that the symmetry type in the $t_f$-aspect is $\mathrm{SO(even)}$.

\begin{thm}\label{thm: t-aspect, 2} Let $T$, $M$ and  $h_{T, \shskip  M} $ be as in Theorem {\rm\ref{thm: t-aspect, 1}}.  Let  $\phi_1 $ and $\phi_2$ be even Schwartz functions  with the supports of $\widehat{\phi}_{1}$ and $\widehat{\phi}_{2}$ in $\left(-\frac 1 2, \frac 1 2\right)$. Then 
	\begin{align*}
	\lim_{T \ra \infty} 	\SD_2 (H^{\star} (1), \phi_1, \phi_2 ; h_{T, \shskip  M}) \hskip -1pt = \hskip -1pt \int \hskip -4pt \int   \phi_1 (x_1) \phi_2 (x_2) W_2 (\mathrm{SO(even)}) (x_1, x_2) d x_1 d x_2. 
	\end{align*}
\end{thm}

\subsection*{Notation} By $B \Lt D$ or $B = O (D)$ we mean that $|B| \leqslant c D$ for some constant $c > 0$, and by $B \asymp D$  we mean that  $B \Lt D$ and $D \Lt B$. We write $B \Lt_{P, \shskip Q, \, \dots} D$ or $B = O_{P, \shskip Q, \, \dots} (D)$ if the implied constant $c$ depends on  $P$, $Q, \dots$. Throughout this article $T > 1$ will be a large parameter, and we say   $B$ is negligibly small if $B = O_A (T^{-A})$ for any $A > 0$. We adopt the usual $\vepsilon$-convention of analytic number theory; the value of $\vepsilon $ may differ from one occurrence to another.

\begin{acknowledgement}
	We thank the referee for his/her careful reading, for pointing out an error, and for several suggestions which helped improve our paper.
\end{acknowledgement}

\section{Preliminaries}

Throughout this paper, fix an imaginary quadratic field $F  $ of class number $h_{F} = 1$. Let $\frO$ be the ring of integers, $\frd$ be the different ideal, and $\frOO^{\times} $ be the unit group.  Let $d_F$ be the discriminant and $\varw_F = |\frOO^{\times}|$. Let $\RN$ and $\Tr$ be the norm and the trace map on $F$.

In general, we shall use Gothic letters $\fra$, $\frm$, $\frn$, $\fp$, $\fq$, $\frv$ ... to denote {\it nonzero} fractional (mostly integral) ideals of $F$. As convention, let $\fp$ always stand for a prime ideal.  For a nonzero integral ideal $\frn $, let $\RN (\frn) = |\frO / \frn |$ be the norm of $\frn$. 

For $n_1$, $n_2  \in \frO$ and $c \in \frd \smallsetminus\{0\}$, define the Kloosterman sum \begin{equation}
\label{1eq: Kloosterman} S (n_1, n_2; c) = \sumx_{\sstyle a  (\mod c \frd\-) } e \lp \mathrm{Tr} \lp \frac {n_1 a + n_2 \widebar a }    {  c}  \rp \rp,
\end{equation}
where $\sumx$ means that $a$ runs over representatives of $( \frO / c \frd\-)^{\times} $ and $a \widebar a \equiv 1 (\mod c \frd\-)$. 
We have Weil's bound  
\begin{equation}
\label{1eq: Weil}
S (n_1, n_2; c)   \Lt  \RN(n_1  , n_2   , c \frd\-)^{\frac 1 2}  \RN (c)^{\frac 1 2 + \sepsilon},
\end{equation}  
where the brackets $(\shskip  \cdot \shskip , \shskip \cdot \shskip, \shskip \cdot \shskip)$ denote greatest common divisor.

\subsection{Automorphic forms on $\BH^3$} 

In this section, we recollect some basic notions in the theory of automorphic forms on  $\BH^3$. We refer the reader to \cite{EGM} or \cite{B-Mo2} for further details. 

\vskip 5 pt

\subsubsection{The three-dimensional hyperbolic space}

We let $$\BH^3 = \left\{ w = z + j r = x + i y + j r : x, y, r \text{ real}, r > 0 \right\}  
$$ denote the three-dimensional hyperbolic space, with the action of $\GL_2 (\BC)$ or $\PGL_2 (\BC)$($= \PSL_2 (\BC)$) given by 
\begin{align*}
z (g \hskip -1pt \cdot \hskip -1pt w) = \frac { (az+ b) (\overline {c} \overline z + \overline d) \hskip -1pt + \hskip -1pt a \overline c r^2 } {|cz+d|^2 \hskip -1pt + \hskip -1pt |c|^2 r^2} , \hskip 5 pt r (g \hskip -1pt \cdot \hskip -1pt w) = \frac {r |\det g| } {{|cz+d|^2 \hskip -1pt + \hskip -1pt |c|^2 r^2}}, \quad g = \begin{pmatrix}
a \hskip -1pt & b \\ 
c \hskip -1pt & d \\
\end{pmatrix} \hskip -1pt,
\end{align*}  
while the action of $\GL_2 (\BC)$ on the boundary $\partial \BH^3 =   \BC \cup \{\infty \}$ is by the M\"obius transform.
$\BH^3$ is equipped with the $\GL_2 (\BC)$-invariant hyperbolic metric $    ( d x^2 + d y^2 + d r^2) / r^2$ and hyperbolic measure $    d x\, dy\, d r / r^3$.  The associated hyperbolic Laplace--Beltrami operator is given by $\varDelta =    r^2 \lp \partial^2/\partial x^2 + \partial^2/\partial y^2 + \partial^2/\partial r^2  \rp -   r \partial /\partial r$.

\vskip 5 pt

\subsubsection{Hecke congruence groups}

%Let $\Gamma = \GL_2 (\frO)$. 
For  a nonzero integral ideal $\fq$, define the Hecke congruence group $\Gamma_0 (\fq)$ of $\GL_2 (\frO)$ as follows,
\begin{align*}%\label{1eq: congruence group, SL}
\Gamma_0 (\fq) = \left\{ \begin{pmatrix}
a & b \\ 
c & d \\
\end{pmatrix} \in \GL_2 (\frO) : c \equiv 0\, (\mod \fq) \right\}. %/ \left\{ \epsilon I_2 : \epsilon \in \frOO^{\times} \right\}.
\end{align*}
In this paper, we assume that $\fq$ is square-free. %and relatively prime to $\frd$.  

\vskip 5 pt

\subsubsection{Maass cusp forms}
Let  $ \mathscr{B} (\fq) $ be an orthogonal basis of Maass cusp forms  in the  $L^2$-cuspidal spectrum  of the  Laplace--Beltrami operator  $\varDelta$ on $\Gamma_0 (\fq)  \backslash \BH^3$. For $f \in \mathscr{B} (\fq) $, let $\|f \|$ denote the $L^2$-norm of $f$ on  $\Gamma_0 (\fq)  \backslash \BH^3$,
\begin{align*}
\|f\|^2 = \int_{ \Gamma_0(\fq) \backslash \BH^3} |f (w)|^2 \frac {d x\, dy\, d r}  {r^3} .
\end{align*} 
For $f   $ with Laplacian eigenvalue $    1 + 4 t_f^2$, we have the Fourier expansion (cf. \cite[Theorem 3.3.1]{EGM})
\begin{align*}
f  (z, r) = \sum_{ m \, \in \frd^{-1} \smallsetminus \{0\} } \rho_f (m) r K_{2 i t_f} (4 \pi |m| r) e (\Tr (m z)).
\end{align*}
Since $ \displaystyle   \begin{pmatrix}
\epsilon  & \\ & 1 
\end{pmatrix}  \in \Gamma_0 (\fq) $ for any $\epsilon \in \frOO^{\times}$,  $\rho_f (m)$ only depends on the ideal $\frm = (m)$, and we may therefore set $\rho_f (\frm) = \rho_f (m)$. 
According to the Kim--Sarnak bound in \cite{Blomer-Brumley}, we know that $t_f $ is either real or purely imaginary with $  | \mathrm{Im} (t_f)|  \leqslant \frac{7}{64}.$

%For any $\epsilon \in \frOO^{\times}$, we have $\displaystyle \widebar{\epsilon}  \hskip -0.5 pt \cdot \hskip -2 pt  \begin{pmatrix}
%\epsilon^2 & \\ & 1 
%\end{pmatrix}  \in \Gamma_0 (\fq)$, we infer that $f $ is invariant under the action of $\displaystyle \begin{pmatrix}
%\epsilon^2 & \\ & 1 
%\end{pmatrix}$ and hence $ \rho_f (\epsilon^2 n) = \rho_f (n)$. Define $\frOO^{\times 2} = \{  \epsilon^2 : \epsilon \in \frO \}$. Since $\frOO^{\times}$ is cyclic and of even order, $\frOO^{\times} / \frOO^{\times 2}$ is of order $2$.  We say that a Maass cusp form for $\Gamma_0 (\fq) $ is   {even} or odd  if  it is an eigenfunction of the action of $\displaystyle \begin{pmatrix}
%\epsilon & \\ & 1 
%\end{pmatrix}$  with eigenvalue $1$ or $-1$ for $\epsilon \in \frO \smallsetminus \frOO^{\times 2}$, respectively. Note that $f $ is even if and only if $\rho_f (\epsilon n) = \rho_f (n)$ for all $\epsilon \in \frOO^{\times}$ and $n \in \frd\- \smallsetminus \{0\}$; we may thus write $ \rho_f (\frn) = \rho_f ( n) $ for $\frn = (n)$. Subsequently, we shall focus only on even Maass cusp forms. 

\vskip 5 pt

\subsubsection{Hecke--Maass newforms and their $L$-functions}\label{section: newforms}

For a nonzero integral ideal $\frn = (n) $,  we define the Hecke operator $T_{\frn}$ by
\begin{equation*}
T_{\frn} f (\tw) = \frac 1 {\varw_F \sqrt {\RN (n)}} \sum_{ a d = n} \ \sum_{ b (\mod d)} f \lp \begin{pmatrix}
a & b \\
& d
\end{pmatrix}
\tw \rp.
\end{equation*}
Hecke operators commute with each other as well as  the Laplacian operator. 

Let $H^{\star} (\fq) $ %= \{ f_{j, q} \}_{j \geqslant 1}$ 
be the   set of {\it primitive}  newforms for $\Gamma_0 (\fq)$ %(newforms for $\Gamma'_0 (q)$) 
which are eigenfunctions of all the $T_\frn$. A form $f$ is called primitive if $\rho_f (\frd\-) = 1$. Let $\lambdaup_f (\frn) $ denote the   Hecke eigenvalue of $T_\frn$ for $f $, then % the formula above yields the relation 
\begin{equation*}%\label{1eq: relation Fourier Hecke}
\rho_f (\frd^{-1} \frn ) =  \lambdaup_f ( \frn).  
\end{equation*}
For a newform $f$ in $H^{\star} (\fq) $,  we have  the Hecke relation
\begin{align}\label{1eq: Hecke rel, 2}
\lambdaup_f (\frn_1) \lambdaup_f (\frn_2) =  \mathop{\sum_{ \fra | \frn_1, \shskip \fra| \frn_2}}_{(\fra, \shskip \fq) = (1) } \lambdaup_f \big(\frn_1 \frn_2  \fra^{-2}\big),
\end{align}
and
\begin{align}\label{1eq: lambda p, p|q}
\lambdaup_f (\fp)^2 = \frac 1 {\RN (\fp)}, \quad \quad \fp | \fq .
\end{align} 
Moreover, for $ \fp \hskip -1pt \nmid \hskip -1pt \fq $, if we let  $ \alphaup_f (\fp) +  \betaup_f (\fp) = \lambdaup_f (\fp) $ and $ \alphaup_f (\fp) \betaup_f (\fp) = 1 $, then we have
 the Kim--Sarnak bound in \cite{Blomer-Brumley}:
\begin{align} \label{Ramanujan}
|\alphaup_f(\mathfrak{p})|, \, |\betaup_f(\mathfrak{p})| \leqslant  \RN(\mathfrak{p})^{\frac{7}{64}}, \quad \quad \fp \hskip -1pt \nmid \hskip -1pt \fq. 
\end{align}

For a newform $ f    \in  H^{\star} (\fq) $, its associated $L$-function $L (s, f )$ is defined by 
\begin{align*}%\label{1eq: L (f)}
L (s, f  ) %= \frac 1 {16} \underset{n_1, n_2 \in \frO \smallsetminus \{0\} }{\sum \sum} \frac {\lambdaup_f (n_2) A (n_1, n_2) } {\left|n_1^2 n_2\right|^{2 s} } 
= \underset{\frn \shskip \subset \frOO  }{\sum} \frac {  \lambdaup_f   (\frn ) } {\RN  ( \frn   )^{  s} }, \quad \quad \Re  (s) > 1. 
\end{align*} 
Its Euler product is
\begin{align*}
L(s, f )  = \prod_{\fp | \fq}\left( 1-  {\lambdaup_f(\fp)}{\RN(\fp)^{-s}}  \right)^{-1} \prod_{\fp \shskip \nmid \fq} \left(   1 -  {\lambdaup_f(\fp)}{\RN(\fp)^{-s}} + \RN (\fp)^{-2s} \right)^{-1}.
\end{align*}
The $L$-function $L (s, f)$ analytically continues to an entire function of the complex
plane. The completed $L$-function
\begin{equation*}
\Lambda (s,f ) =  (2\pi)^{-2s} \Gamma (s - i t_f) \Gamma   (s + i t_f ) |d_F|^s \RN(\fq)^{s/2}  L(s,f )
\end{equation*}
satisfies the functional equation:
\begin{equation*}
\Lambda (s,f  ) = \epsilon_f \Lambda (1-s, f ),
\end{equation*}
where $\epsilon_f=   \mu(\fq) \lambdaup_f (\fq) \sqrt {\RN(\fq)} = \pm 1$. 
Note that the local factor of $L (s, f )$ at $\fp \nmid \fq$ factors further as 
\begin{align*}
\lp 1 - \alphaup_f (\fp) \RN(\fp)^{-s} \rp^{-1} \lp 1 - \betaup_f (\fp) \RN(\fp)^{-s} \rp^{-1}.
\end{align*}
%where $  \alphaup_f (\fp) $, $ \betaup_f (\fp)$ are complex numbers with $ \alphaup_f (\fp) +  \betaup_f (\fp) = \lambdaup_f (\fp) $ and $ \alphaup_f (\fp) \betaup_f (\fp) = 1 $. 

\subsection{The spectral Kuznetsov trace formula  for $\Gamma_0  (\fq) \backslash \BH^3$}\label{sec: Kuznestov}

For $t $ real, define the Bessel function (cf. \cite[\S 15.3]{Qi-Bessel})
\begin{equation}\label{0eq: defn of Bessel}
\boldJ_{it} (z) = \frac {2 \pi^2} {\sin (2\pi i t)}  \big( J_{- 2 it   } \lp 4 \pi    z \rp J_{- 2 it    } (  4 \pi   {\widebar z} )    - J_{ 2 it   } \lp 4 \pi    z \rp J_{ 2 it }  ( 4 \pi   {\widebar z} )   \big) ,
\end{equation}
in which $J_{\varnu} (z)$ is the Bessel function of the first kind. 
This  is a well-defined {\it even} function in the sense that the   expression on the right of \eqref{0eq: defn of Bessel} is independent on the choice of the argument of $z$ modulo $ \pi$. It is understood that in the non-generic case when  $t = 0$  the right-hand side should be replaced by its limit.

Let $h (t)$ be a weight function in the space $ \mathscr{H} (S, N)  $ defined as in Definition \ref{def: test functions}.\footnote{According to \cite{B-Mo2}, the Kuznetsov trace formula is valid for an enlarged space of weight functions. Thus our results should still be valid if \eqref{0eq: bounds for h(t)} is replaced by the Schwartz condition.}    
For nonzero integral ideals $ \frn_1,$ $\frn_2 $, define
\begin{align}\label{1eq: Kuznetsov} 
\Delta_{\fq} (\frn_1, \frn_2; h) & =  \sum_{f \in  \mathscr{B} (\fq)} \hskip -1pt \omega_f    { h  ( t_f ) \rho_f  (\frd\- \frn_1)   \overline {\rho_f (\frd\- \frn_2)} },  \\
\label{1eq: Xi} \Xi_{\fq} (\frn_1, \frn_2; h) & 
= \frac 4 {  \pi \varw_F \hskip -1pt \sqrt{|d_F|} } \sum_{ \frv | \fq } \hskip -1pt   \int_{-\infty}^{\infty} \hskip -2pt  \omega_{\frv} (t) h (   t ) 
\eta \hskip -1pt \lp \frn_1, \hskip -1pt \tfrac 1 2 + i t \rp  \eta \hskip -1pt \lp \frn_2 , \hskip -1pt \tfrac 1 2 - i t \rp  \hskip -1pt  d t , 
\end{align}
in which 
\begin{align}
\label{1eq: omegas} \omega_f = \frac {  t_f }   { \sinh (2 \pi t_f) \|f\|^2 },  \hskip 10 pt \omega_{\frv} (t) =  \frac {  8 \pi  |   \zeta_{F, \shskip \fq} (1+2it) |^2  } { |d_F| \RN(\mathfrak v \mathfrak q) |  \zeta_F (1+2it) |^2  }   , \quad \frv | \fq,
\end{align}
\begin{align}\label{1eq: eta}
\eta (\frn, s) =  \sum_{ \sstyle \fra |\frn  }  \RN (\frn \fra^{-2})^{ s - \frac 1 2} , \quad \zeta_F (s) = \sum_{ \frn \shskip \subset \frOO } \frac 1 {\RN(\frn)^{s}}, \quad \zeta_{F, \shskip   \fq}  (s) = \prod_{\sstyle \fp | \fq   }  \frac 1   {1 - \RN(\fp)^{- s}} .
\end{align} 
The spectral Kuznetsov trace formula for $\Gamma_0 (\fq)$ follows from \cite[Theorem 11.3.3]{B-Mo2}  (see also \cite{BM-Kuz-Spherical} and \cite{Venkatesh-BeyondEndoscopy}) and computations for the Eisenstein series similar to those in \cite[\S 3]{CI-Cubic} or \cite[\S 3.1]{Qi-Gauss}. 
\begin{prop}[Kuznetsov trace formula for $\Gamma_0 (\fq)$]\label{prop: Kuznetsov} We have
	\begin{align*} 
	\Delta_{\fq} (\frn_1, \frn_2; h) + \Xi_{\fq}   (\frn_1, \frn_2 &; h)   \hskip -1pt = \hskip -1pt	\frac { 8  } {   \pi^2 \hskip -1pt \sqrt{|d_F|} }  \delta_{  \frn_1,   \frn_2}  \, H \\
	& \hskip -1pt    + \hskip -1pt \frac 4 {   \pi^2 |d_F| } \hskip -1pt \sum_{ \epsilon \hskip 0.5 pt \in \frOO^{\times} / \frOO^{\times  2} } \, \sum_{  c \shskip \in \fq \frd \smallsetminus \{0\} } \frac {S ( \epsilon n_1,  n_2; c)} {\RN (c \frd\-)}  H \bigg(   \frac {\textstyle \sqrt {\epsilon n_1 n_2} } { c } \bigg) \hskip -1pt , 
	\end{align*}
\footnote{The constants in the formula are adopted from (16) in    \cite[Proposition 1]{Venkatesh-BeyondEndoscopy}. }	where 
	\begin{align}\label{1eq: H's} H = \int_{-\infty}^{\infty} h (  t  )    t^2  d t,  \quad H (z) =  \int_{-\infty}^{\infty} h (   t ) \boldsymbol{J}_{i t} (z)   t^2 d t, 
	\end{align} 
	$\delta_{  \frn_1, \shskip  \frn_2} $ is the Kronecker $\delta$-symbol, $\frOO^{\times 2} = \left\{ \epsilon^2 : \epsilon \in \frO^{\times} \right\}$, $  \frn_1 = (n_1) $, $  \frn_2 = (n_2) $, and $S (\epsilon n_1,  n_2, c)$ is the Kloosterman sum defined by \eqref{1eq: Kloosterman}. 
\end{prop}

Next we express $\Delta_{\fq} (\frn_1, \frn_2; h)$ in terms of  newforms. 

 For $\fq = \frv \mathfrak{w} $ and   primitive newform $f \in H^{\star} (\frv)$, let $S (\mathfrak{w} ; f) $ denote  the linear space spanned by the forms $f_{| \hskip 0.5 pt d} (z, r) = f (d z, |d| r)$, with $d \in \mathfrak{w} \smallsetminus \{0\}$. The space of    cusp forms  for $\Gamma_0 (\fq)$   decomposes into the orthogonal sum of  $S (\mathfrak{w} ; f) $.  By the calculations in  \cite[\S 2]{ILS-LLZ}, one may construct an orthonormal basis  of $S (\mathfrak{w} ; f) $ in terms of $f_{| \hskip 0.5 pt d}$.    Using this collection of bases as our $\mathscr{B} (\fq)$, the sum $\Delta_{\fq} (\frn_1, \frn_2; h)$ in \eqref{1eq: Kuznetsov} can be arranged into a sum over the primitive newforms in $  H^{\star} (\frv)$ for all $\frv | \fq$. 
To be precise, following \cite[\S 2]{ILS-LLZ} (see also \cite[\S 3.2.2]{Qi-Gauss}), for $(\frn_1 \frn_2, \fq) = (1)$, we may derive the formula 
\begin{equation}\label{2eq: Delta, 2}
\begin{split}
\Delta_{\fq} (\frn_1, \frn_2; h) =  \frac 1 {\RN (\fq)} \sum_{ \frv | \fq   }   & \sum_{f \in H^{\star} (\frv) } \omega_{f }^{\star}   h  ( t_f ) 
\lambdaup_f  (\frn_1)   \overline {\lambdaup_f (\frn_2)} ,
\end{split}
\end{equation}
where
\begin{align}\label{2eq: omega*f}
\omega^{\star}_{f }  = \frac { 64 \pi^2 Z_{\fq} (1,   f  ) } { {|d_F|}^2 \zeta_F (2)   Z (1, f  ) }, \quad \quad f \in H^{\star} (\frv) , \, \frv | \fq, 
\end{align}
with 
\begin{align}\label{2eq: Z(s, f)}
Z (s,  f  )  =   \sum_{\frn \shskip \subset \frOO }  \frac {\lambdaup_f (\frn^2)} {\RN(\frn)^{s}}, \hskip 10pt Z_\fq (s,  f  ) =  \sum_{\frn | \fq^{\infty} }  \frac {\lambdaup_f (\frn^2)} { \RN(\frn)^{s}}.
\end{align}  
Here $\zeta_F (2)$ arises from $ \mathrm{Vol} (\PSL_2 (\frO) \backslash \BH^3) = |d_F|^{3/2} \zeta_{F} (2) / 4 \pi^2 $ as in \cite[\S 7.1, Theorem 1.1]{EGM} (note that $\PSL_2 (\frO)$ is of index $2$ in $\PGL_2 (\frO)$). 

For square-free $\fq$, we introduce  
\begin{align} 
\Delta_{\fq}^{\star} (\frn_1, \frn_2; h)    =     & \sum_{f \in H^{\star} (\fq) } \omega_{f }^{\star}   h  ( t_f ) 
\lambdaup_f  (\frn_1)   \overline {\lambdaup_f (\frn_2)} , \\
\label{2eq: Xi*}\Xi _{\fq}^{\star} (\frn_1, \frn_2; h) = &  \int_{-\infty}^{\infty} \hskip -2pt  \omega_{\fq}^{\star} (t) h (   t ) 
\eta \hskip -1pt \lp \frn_1, \hskip -1pt \tfrac 1 2 + i t \rp  \eta \hskip -1pt \lp \frn_2 , \hskip -1pt \tfrac 1 2 - i t \rp  \hskip -1pt  d t , 
\end{align}
with
\begin{align}\label{2eq: omega*}
\omega_{f }^{\star} = \frac { 64 \pi^2 Z_{\fq} (1,   f  ) } { {|d_F|}^2 \zeta_F (2)   Z (1, f  ) }, \quad \omega_{\fq}^{\star} (t) =   \frac {  32    |   \zeta_{F, \shskip \fq} (1+2it) |^2  } { \varw_F |d_F|^{3/2} \RN(\fq ) |  \zeta_F (1+2it) |^2  }, \quad f \in H^{\star} (\fq).
\end{align}
Then \eqref{2eq: Delta, 2} may be further written as
\begin{align*}
\Delta_{\fq} (\frn_1, \frn_2; h) =  \frac 1 {\RN (\fq)} \sum_{ \fq = \frv \mathfrak{w}  } \sum_{\   \mathfrak{l} | \mathfrak{w}^{\infty} } \frac 1 {\RN (\mathfrak l)} \Delta_{\frv}^{\star} (\frn_1 \mathfrak{l}^2 , \frn_2; h) ,
\end{align*}
for  $(\frn_1 \frn_2, \fq) = (1)$, and similarly \eqref{1eq: Xi} as
\begin{align*}
\Xi_{\fq} (\frn_1, \frn_2; h) =  \frac 1 {\RN (\fq)} \sum_{ \fq = \frv \mathfrak{w}  } \sum_{\   \mathfrak{l} | \mathfrak{w}^{\infty} } \frac 1 {\RN (\mathfrak l)} \Xi_{\frv}^{\star} (\frn_1 \mathfrak{l}^2 , \frn_2; h) .
\end{align*} 
From M\"obius inversion, we arrive at the following lemma. 

\begin{lem}\label{lem: Delta* = Delta}
	For $(\frn_1 \frn_2, \fq) = (1)$, we have 
	\begin{align*}%\label{2eq: Delta*=Delta}
	\Delta_{\fq}^{\star} (\frn_1, \frn_2; h) & = \sum_{ \fq = \frv \mathfrak{w}  } \mu (\mathfrak{w}) \RN (\frv) \sum_{\   \mathfrak{l} | \mathfrak{w}^{\infty} } \frac 1 {\RN (\mathfrak l)} \Delta_{\frv}  (\frn_1 \mathfrak{l}^2 , \frn_2; h),  \\
	\Xi_{\fq}^{\star} (\frn_1, \frn_2; h) & = \sum_{ \fq = \frv \mathfrak{w}  } \mu (\mathfrak{w}) \RN (\frv) \sum_{\   \mathfrak{l} | \mathfrak{w}^{\infty} } \frac 1 {\RN (\mathfrak l)} \Xi_{\frv}  (\frn_1 \mathfrak{l}^2 , \frn_2; h), 
	\end{align*}
	where $\mu $ is the M\"obius function for $F$.
\end{lem}

In view of \eqref{2eq: Xi*} and \eqref{2eq: omega*},   $\Xi_{\fq}^{\star} (\frn_1, \frn_2; h)$ may be easily estimated using the lower bound $ |  \zeta_F (1+2it) | \Gt_{ F} 1/ \log  (|t|+3)$.\footnote{This readily follows from some standard arguments due to Landau as in \cite[\S\S 3.9--3.11]{Titchmarsh-Riemann}. We also refer the reader to \cite[Lemma 10]{Mit}. Although only the bound for $\zeta_F' (s) / \zeta_F (s)$ is given in \cite{Mit}, $1 / \zeta_F (s)$ should have the same bound in view of \cite[Theorem 3.11]{Titchmarsh-Riemann}. Actually, \cite[Lemma 11]{Mit} suggests a stronger bound $1 / \zeta_F (1+it) = O \big((\log |t|)^{\frac 2 3} (\log \log |t|)^{\frac 1 3} \big) $ ($|t| \geqslant 3$); the reader may compare this with \cite[(6.15.2)]{Titchmarsh-Riemann}.}

\begin{lem}\label{lem: bounds for Xi}
	We have 
	\begin{align*} 	\Xi_{\fq}^{\star} (\frn_1, \frn_2; h) \Lt_{\, F} \frac {\tau (\frn_1) \tau (\frn_2)} {\varphi (\fq)} K,
	\end{align*}
	where, as usual,  $\tau (\frn)$ is the divisor function for $F$,  $  \varphi (\fq) =   \RN (\fq)  / \zeta_{F, \shskip   \fq} (1) $ is Euler's totient function for $F$, and
	\begin{align}\label{2eq: defn of K}
	K = \int_{-\infty}^{\infty} h (t) \log^2 (|t|+3) d t . 
	\end{align}
\end{lem}

For $(\frn, \fq) = (1) $,  define 
\begin{align}
\Delta_{\fq} (\frn; h) = \sum_{f \in H^{\star} (\fq)} \hskip -1pt \omega_f^{\star}  h  ( t_f ) \lambdaup_f (\frn) . 
\end{align}

\begin{cor}\label{cor: Delta (n) =}
Let $\fp \hskip -1pt \nmid \hskip -1pt \fq$. For $\frn = (1), $ $\fp$ or $\fp^2$,  we have
	\begin{equation}\label{2eq: Kuznetsov, 2}
	\begin{split}
	\Delta_{\fq} (\frn; h) =    \varphi (\fq)  \Delta (h) \cdot \delta_{  \frn, \shskip  (1)}   + {\widetilde{\mathrm{KB}}_{\fq} ( \frn; h)}
	  + O_F \bigg(\frac K {\varphi (\fq)}  \bigg)   ,
	\end{split}
	\end{equation}
	where 
	\begin{align}\label{2eq: tilde KBq}
\Delta  (h) = \frac {8  H } {\pi^2 \sqrt{|d_F|}}, \quad \widetilde{\mathrm{KB}}_{\fq} (  \frn; h) = \sum_{ \fq = \frv \mathfrak{w}  } \mu (\mathfrak{w}) \RN (\frv) \sum_{\   \mathfrak{l} | \mathfrak{w}^{\infty} } \frac {\mathrm{KB}_{\fq} (\mathfrak l^2 \frn; h)} {\RN (\mathfrak l)} ,
	\end{align}
	with
	\begin{align}\label{2eq: KBq}
\mathrm{KB}_{\fq} (\frn; h) = \frac 4 {   \pi^2 |d_F| } \hskip -1pt \sum_{ \epsilon \hskip 0.5 pt \in \frOO^{\times} / \frOO^{\times  2} } \, \sum_{  c \shskip \in \fq \frd \smallsetminus \{0\} } \frac {S ( \epsilon n , 1  ; c)} {\RN (c)}  H \bigg(   \frac {\textstyle \sqrt {\epsilon n   } } { c} \bigg),
	\end{align}	
	for $   \frn = (n) ${\rm;} here $H$, $H (z)$ and $K$ are defined as in {\rm\eqref{1eq: H's}} and {\rm\eqref{2eq: defn of K}}.   
\end{cor}

\begin{proof}
	Apply Lemma \ref{lem: Delta* = Delta} and \ref{lem: bounds for Xi} with $\frn_1 = \frn$ and $\frn_2 = (1)$, along with the Kuznetsov trace formula in Proposition \ref{prop: Kuznetsov}. 
\end{proof}

\begin{cor}\label{cor: Delta (n1, n2) =}
	For $\frn_1 = (1), $ $\fp_1$ or $\fp_1^2$ and $\frn_2 = (1), $ $\fp_2$ or $\fp_2^2$, we have
	\begin{equation}\label{2eq: Kuznetsov, 2.2}
	\begin{split}
	\Delta_{(1)} (\frn_1, \frn_2; h) =       \Delta (h) \cdot \delta_{  \frn_1, \shskip \frn_2}   + {{\mathrm{KB}}_{(1)} ( \frn_1, \frn_2; h)}
	+ O_F  (K)  ,
	\end{split}
	\end{equation}
	where  $K$ and $\Delta  (h)$   are defined as in  {\rm\eqref{2eq: defn of K}} and {\rm\eqref{2eq: tilde KBq}}, and
	\begin{align}\label{2eq: KBq, 2}
	\mathrm{KB}_{(1)} (\frn_1, \frn_2; h) = \frac 4 {   \pi^2 |d_F| } \hskip -1pt \sum_{ \epsilon \hskip 0.5 pt \in \frOO^{\times} / \frOO^{\times  2} } \, \sum_{  c \shskip \in \frd \smallsetminus \{0\} } \frac {S (  \epsilon n_1 ,  n_2 ; c)} {\RN (c)}  H \bigg(   \frac {\textstyle \sqrt {\epsilon n_1 n_2   } } { c} \bigg),
	\end{align}	
	with $  \frn_{1} = (n_1) $ and $\frn_2 = (n_2)$.   
\end{cor}

 \subsection{Stationary phase} 
 
 The following lemma is an improvement of   {\rm\cite[Lemma {\rm8.1}]{BKY-Mass}} (cf.    \cite[Lemma A.1]{AHLQ}). It will be used to show that certain exponential integrals are negligibly small in the absence of stationary phase.
 
 \begin{lem}\label{lem: staionary phase, dim 1, 2}
 	Let $\tw (x)$ be a smooth function supported on $[  a, b]$ and $f (x)$ be a real smooth function on  $[  a, b]$. Suppose that there
 	are   parameters $Q, U,   X, Y,  R > 0$ such that
 	\begin{align*}
 	f^{(i)} (x) \Lt_{ \, i } Y / Q^{i}, \hskip 10pt \tw^{(j)} (x) \Lt_{ \, j } X / U^{j},
 	\end{align*}
 	for  $i \geqslant 2$ and $j \geqslant 0$, and
 	\begin{align*}
 	| f' (x) | \geqslant R. 
 	\end{align*}
 	Then for any $A \geqslant 0$ we have
 	\begin{align*}
 	\int_a^b e (f(x)) \tw (x)  {d} x \Lt_{ A} (b - a) X \bigg( \frac {Y} {R^2Q^2} + \frac 1 {RQ} + \frac 1 {RU} \bigg)^A .
 	\end{align*}
 	%	where the implied constant depends only on $A$ and  those in the estimates for the derivatives of $f (x)$ and $w (x)$. 
 \end{lem}

	\section{Analysis of Bessel integrals}

Let $h (t),   h_{T, \shskip  M} (t) \in \mathscr{H}^+ (S, N)$ be as in Definition \ref{def: test functions}. Let   $H (z)$, respectively  $H_{T, \shskip  M} (z)$, be the Bessel integral as defined in \eqref{1eq: H's} for $h (t)$, respectively for $h_{T, \shskip  M} (t)$. Set $M = T^{\mu}$ for  $0 < \mu < 1$. 

Before reading this section, we refer the reader to Appendix \S \ref{sec: app H+(x)} and \ref{sec: app H-(x)} for the analysis of real Bessel integrals $ H_{T, \shskip M}^{\ssstyle +} (x)$ and $H_{T, \shskip M}^{\ssstyle -} (x)$. 

\subsection{The case $|z| \leqslant 1$}

\begin{lem}\label{lem: H(z), 1}
	%Let $M = T^{\mu}$ with $0 < \mu < 1$.	  
	We have $ H  (z) \Lt_{\, h}    |z|^2 $ for $|z| \leqslant 1$. % with the implied constants depending only on $h $.
\end{lem}

\begin{proof}
	By the definitions in \eqref{0eq: defn of Bessel} and \eqref{1eq: H's}, we have
	\begin{align*}
	H (z) =  - {4 \pi^2} \int_{-\infty}^{\infty} h  (t)  \frac {J_{    2it    } \lp 4 \pi  z \rp J_{  2it    } \lp 4 \pi  { \widebar z} \rp} {\sin \lp 2\pi it   \rp }  t^2  {d}  t.
	\end{align*}
	Note that the Plancherel measure $t^2 d t$ vanishes at $t =0$. 
	Shifting the line of integration to $\Im  ( t) = - \frac 3 4 $  and crossing only the pole at $t = -  \frac 1 2 i   $, we obtain
	\begin{align*}
	H    (z)   =  \, & - \pi^2 h \hskip -1pt \lp - \tfrac 1 2 i \rp J_{1} (4 \pi z) J_{1} (4 \pi  \widebar z) \\
	&  + 4 \pi^2  \int_{-\infty}^{\infty} h  \hskip -1pt \lp   t - \hskip - 1 pt \tfrac 3 4 i   \rp   \frac {J_{  2it + \frac 3 2 } \lp 4 \pi  z \rp J_{  2it + \frac 3 2 } \lp 4 \pi  { \widebar z} \rp} {\cosh   (2\pi  t)      }   \hskip -1pt \lp t - \tfrac 3 4 i  \rp^2 d t.
	\end{align*}
	By   Poisson's integral representation for $J_{\varnu} (z)$
	(cf.  \cite[3.3 (6)]{Watson} or \cite[8.411.4]{G-R}), 
	\begin{align}\label{2eq: Poisson, J}
	J_{\varnu} (z) = \frac {\big(\tfrac 1 2 z\big)^{\varnu}} {\Gamma \big(\varnu + \tfrac 1 2 \big) \Gamma \big(   \tfrac 1 2 \big)} \int_0^\pi e^{i \hskip 0.5 pt z \cos \theta} \sin^{2 \varnu} \theta \,  {d}   \theta, \hskip 10 pt \Re  (\varnu)  > - \tfrac 1 2, 
	\end{align}
	we infer that  
	\begin{align*}
	\left| J_{\varnu} (z) \right| \Lt \frac {   \left|   z  ^{\varnu} \right|  } {\Gamma \big(\varnu + \tfrac 1 2 \big)},  \hskip 20 pt |z| \leqslant 4 \pi.
	\end{align*} 
	Hence,  Stirling's formula yields the estimates
	\begin{align*}%\label{1eq: J for |z|<1}
	J_{1} (4 \pi z) J_{1} (4 \pi  \widebar z) \Lt |z|^2, \hskip 10 pt  
	\frac {J_{  2it + \frac 3 2 } \lp 4 \pi  z \rp J_{  2it + \frac 3 2 } \lp 4 \pi  { \widebar z} \rp} {\cosh   (2\pi  t)      }   \Lt    \lp \frac { |z|  } { |t| + 1   } \rp^{3}    , 
	\end{align*} 
	for $|z| \leqslant 1$. 
	Consequently, in view of \eqref{0eq: bounds for h(t)}, we have
	\begin{align*}%\label{5eq: H(z) for z small}
	H (z) \Lt |z|^2   , 
	\end{align*}
	for $|z| \leqslant 1$.
\end{proof}

\begin{lem}\label{lem: HTM(z), 1}
	Let $M = T^{\mu}$ with $0 < \mu < 1$.	  We have  $ H_{T, \shskip  M} (z) \Lt_{\, h }  M   |z|^2 / T $ for $|z| \leqslant 1$. % with the implied constants depending only on $h $.
\end{lem}

\begin{proof}
	We modify  the proof of Lemma \ref{lem: H(z), 1} by replacing $h (t)$ by $h_{T, \shskip  M} (t)$. An estimation by  \eqref{0eq: bounds for h(t)} and \eqref{0eq: defn of hTM} yields
	\begin{align*}%\label{5eq: H(z) for z small}
	H_{T, \shskip  M} (z) \Lt e^{- T / M} |z|^2  +  \frac { M   |z|^{3} } {T } \Lt   \frac { M   |z|^{2} } {T } ,
	\end{align*} 
	for $|z| \leqslant 1$.
\end{proof}

By shifting the integral contour further, one may prove that $ H_{T, \shskip  M} (z) \Lt   M   |z|^2 / T^{4 \delta -2} $ as long as $  \delta < S $. 

\subsection{The case $|z| > 1$}

\begin{lem}\label{lem: H(z), 2.1}
	We have
	$ H  (z) \Lt_{\, h}   1/ |z| $  for $|z| > 1$.
\end{lem}

\begin{proof}
	This follows immediately from (cf. \cite[Lemma 4.1]{Qi-Gauss})
	\begin{equation*}
	\boldJ_{it} (z) \Lt \frac {t^2+1} {|z|}, \hskip 20pt |z| > 1 .
	\end{equation*}
\end{proof}

\begin{lem}\label{lem: HTM(z), 2.0}
	Let $M = T^{\mu}$ with $ 0 < \mu < 1$. Let $|z| > 1$. We have $ H_{T, \shskip  M} (z) = H^{  \ssstyle \sharp }_{T, \shskip  M} (z) + O_{A}  (T^{-A} ) $ for any $A > 0$, with 
	\begin{align}\label{3eq: H-sharp(z)}
	H_{T, \shskip  M}^{  \ssstyle \sharp}   (  x e^{i\theta} ) =	4 M T^2 \int_0^{2 \pi}   \hskip -1 pt
	\int_{- M^{\ssepsilon} / M}^{M^{\ssepsilon}/ M}   \widehat k  (-2M r / \pi)  
	e  (2 f_T (r, \omega; x, \theta)      )  \, d r \shskip d \omega,
	\end{align} 
	where the phase function is
	\begin{align}\label{3eq: phase function}
	f_T (r, \omega; x, \theta) = Tr/\pi + 2 x (\cosh r \cos \omega \cos \theta - \sinh r \sin \omega \sin \theta).
	\end{align}
	and $k (t) = (Mt / T + 1)^2 h (t)$ is a Schwartz function. 
\end{lem}

\begin{proof}
	%		We shall first work in the polar coordinates. Namely, $z = x e^{i\theta}$.
	First, let us recall the following integral representation in \cite[\S 6.8.3]{Qi-Bessel} (see also \cite[Theorem 12.1]{B-Mo}), \begin{equation*}%\label{1eq: integral representation, 1}
	\boldJ_{i t}(  x      e^{i \theta} )   = 4 \pi   \int_{0}^\infty   J_{0} \hskip -1pt \left( 4 \pi  x  \left| y e^{i\theta} + 1/ y  e^{
		i\theta} \right|  \right) y^{4 i t - 1}  d y.
	\end{equation*} 
	On letting $y = e^r$, we infer that 
	\begin{equation}\label{1eq: integral representation, 2}
	\boldJ_{i t} (  x      e^{i \theta} )   = 4 \pi \int_{- \infty}^{\infty}   J_0 \lp 8 \pi  x  \left| \cosh \lp r +  i \theta \rp \right| \rp e (2 t r/\pi) \, d r.
	\end{equation}  
	%	When $|r| >  T^{\sepsilon}$ and $x      > 1$, we have 	$ J_0 \lp 8 \pi  x       \left| \cosh \lp r +  i \theta \rp \right| \rp \Lt 1 / \sqrt x      e^{ r / 2} $, so only a negligibly small error will be lost if we truncate the integral above at $|r| = T^{\sepsilon}$.  Therefore, up to a negligible error, $H_{T, \shskip  M}   (  x      e^{i \theta} )$ is equal to
	So
	\begin{align*}
	H_{T, \shskip  M}   (  x      e^{i \theta} ) =	8 \pi \int_{-\infty}^{\infty}  
	\int_{-\infty}^{\infty}  \hskip -1 pt t^2 h ((t-T)/M)    
	J_0 \lp 8 \pi  x       \left| \cosh \lp r +  i \theta \rp \right| \rp e (2 t r/\pi)  \, d r \, d\shskip t.
	\end{align*}
	This double integral is absolutely convergent as $ J_0 \lp 8 \pi  x       \left| \cosh \lp r +  i \theta \rp \right| \rp \Lt 1 / \sqrt x      e^{ r / 2} $ for $x > 1$ and $r > 1$. 
	On changing the variable from $t $ to $M t + T$, this integral turns into
	\begin{align*}
	8 \pi M T^2 \int_{- \infty}^{\infty}  
	\int_{- \infty}^{\infty}   k (t)   e ({2 (Mt + T) r}/\pi)    J_0 \lp 8 \pi  x       \left| \cosh \lp r +  i \theta \rp \right| \rp   d r \, d\shskip t,
	\end{align*}
	with 
	\begin{align*}
	k (t) = (Mt / T + 1)^2 h (t) . 
	\end{align*}
	By \eqref{0eq: bounds for h(t)}, we have a similar estimate 
	\begin{align*} 
	k (t+ i \sigma)  \Lt e^{- \pi |t|} (|t|+1)^{2-N}. 
	\end{align*}
	It suffices to know that $k (t) $ is a {\it Schwartz} function (this is because the derivatives of $k (t)$ also satisfy the above estimate   by Cauchy's integral formula). 
	By interchanging the order of integration, we find that the integral above is equal to
	\begin{align*}
	8 \pi M   T^2  
	\int_{- \infty}^{\infty} \widehat k \lp -   { 2 M r} / {\pi} \rp  e(2Tr / \pi)   J_0 \lp 8 \pi  x       \left| \cosh \lp r +  i \theta \rp \right| \rp  d r, 
	\end{align*}
	where $ \widehat k (r) $
	is the Fourier transform of $ k (t)$. Now $\widehat k (r) $ is also Schwartz.  
	Hence, the integral may be effectively truncated at $|r| = M^{\sepsilon} / M$, and we need to consider % we may extend the integral above onto the entire real line along with a negligible error term. Consequently, we need to consider
	\begin{equation*}%\label{eq: def of M, 1}
	8 \pi M   T^2  
	\int_{- M^{\ssepsilon} / M}^{M^{\ssepsilon}/ M} \widehat k \lp -   { 2 M r} / {\pi} \rp  e(2Tr / \pi) J_0 \lp 8 \pi  x       \left| \cosh \lp r +  i \theta \rp \right| \rp    d r.
	\end{equation*}
	Finally, we apply the Bessel integral representation of  $J_0 (2\pi x)$  (cf. \cite[\S 2.2]{Watson}),
	\begin{equation*}%\label{1eq: Bessel integral J0}
	J_0 ( 2 \pi x ) = \frac 1 {2 \pi } \int_0^{2 \pi} e \lp{    x \cos \omega}\rp d \omega, 
	\end{equation*}  
	so that the integral above turns into  $H_{T, \shskip  M}^{  \ssstyle \sharp}   (  x e^{i\theta} )$ defined by \eqref{3eq: H-sharp(z)} and \eqref{3eq: phase function}. 	
\end{proof}

\begin{lem}\label{lem: HTM(z), 2}
	Let $M = T^{\mu}$ with $ 0 < \mu < 1$. Suppose that $|z| > 1$.    Fix a constant $0 < c < 1$. 
	
	{\rm (1).} When  $|\Im (z)| \leqslant c T / 2 \pi  $ and $|\Re (z)| \Lt T M^{1-\sepsilon}  $, we have $H_{T, \shskip  M} (z) = O_{A} (T^{- A})$. % for any $A \geqslant 0$.
	
	{\rm (2).}  Suppose that $\mu > \frac 1 2$.  When $c T / 2 \pi < |\Im (z)| \Lt M^{2+\sepsilon}$ and $|\Re (z) | \Lt M^{2+\sepsilon}$, we have 
	\begin{align}\label{app: main bound for H(z)}
	H_{T, \shskip  M}   (z) \Lt M T^{2 + \sepsilon} / |z| ,
	\end{align}
	and $H_{T, \shskip  M} (z) = O_{A} (T^{- A})$ if
	\begin{align}\label{3eq: Re z small}
|\Re (z)| \Lt |z|^2 / M^{1-\sepsilon} T .
	\end{align} 
\end{lem}

\begin{proof}%[Proof of Lemma \ref{lem: HTM(z), 2}]
	For notational simplicity, we shall work in the Cartesian coordinates, $z = x+iy$. In view of Lemma \ref{lem: HTM(z), 2.0}, we need to consider 
	\begin{align}\label{3eq: H-sigma}
	H^{\ssstyle \sharp}_{T, \shskip  M} (x+iy) = 4 M T^2 \int_0^{2 \pi}   \hskip -1 pt
	\int_{- M^{\ssepsilon} / M}^{M^{\ssepsilon}/ M}  \widehat {k} (-2 M r/\pi)
	e  (2 f_T (r, \omega; x, y)      )  \, d r \shskip d \omega,
	\end{align}
	where 
	\begin{align}\label{3eq: phase function, Cartesian}
	f_T (r, \omega; x, y) = Tr/\pi + 2 x \cosh r \cos \omega  - 2 y \sinh r \sin \omega,
	\end{align} 
	for which 
	\begin{align*}
	(\partial / \partial r) f_T (r, \omega; x, y) & = T/\pi + 2 x \sinh r \cos \omega  - 2 y \cosh r \sin \omega, \\
	(\partial / \partial \omega) f_T (r, \omega; x, y) & = - 2 x \cosh r \sin \omega   - 2 y \sinh r \cos \omega  .
	\end{align*}
	
	(1).	In the first case that  $ |y| \leqslant cT/ 2 \pi $ ($c < 1$) and  $ |x| \Lt T M^{1-\sepsilon} $, note that
	\begin{align*}
	(\partial / \partial r) f_T (r, \omega; x, y) \geqslant  T/\pi - 2 |x| |\sinh r|  -  2 |y| \cosh r  , 
	\end{align*}
	so	it is clear that $ (\partial / \partial r) f_T (r, \omega; x, y) \Gt T $ for $|r| \leqslant M^{\sepsilon} / M$. Along with the assumption $M = T^{\mu} \leqslant T^{1-\sepsilon}$ ($0 < \mu < 1$), Lemma \ref{lem: staionary phase, dim 1, 2}, with $Y = T M^{1-\sepsilon}$, $Q = 1$, $U = 1/M$ and $R=T$, implies that the integral $ H_{T, \shskip  M}^{  \ssstyle \sharp }   (  x   + iy )$ in \eqref{3eq: H-sigma} is negligibly small (for applying Lemma \ref{lem: staionary phase, dim 1, 2} it would be more rigorous if we had used a smooth truncation at $|r| = M^{\sepsilon} / M$). 
	
	(2).	In the second case that $  cT/ 2 \pi < |y| \Lt  M^{2+\sepsilon} $ and $ |x| \Lt M^{2+\sepsilon}  $, Lemma \ref{lem: staionary phase, dim 1, 2}, with $Y= M^{2+\sepsilon}$, $Q = 1$, $U = 1/M$ and $R = M^{1+\sepsilon}$, would again imply that  the $r$-integral in \eqref{3eq: H-sigma} is negligibly small unless 
	\begin{align}\label{app: range of omega}
	\left| T / \pi - 2 y \sin \omega   \right| \leqslant M^{1+\sepsilon} .
	\end{align}
	Note here that $\cosh r = 1 + O (M^{\sepsilon}/ M^2)$ and that $|y \sin \omega| M^{\sepsilon} / M^2 = O (M^{\sepsilon})$. 
	We may therefore (smoothly) truncate the $\omega $-integral at $ |\sin \omega - T / 2\pi y | = M^{1+\sepsilon} / 2 |y| $, and we infer that 
	\begin{align}\label{app: bound for H sharp}
	H_{T, \shskip  M}^{ \ssstyle \sharp}   (  x   + iy ) \Lt M T^{2 + \sepsilon} / |y|. 
	\end{align}
	For $\omega$ satisfying \eqref{app: range of omega} and $|r| \leqslant M^{\sepsilon} / M$, we have
	\begin{align*}
	| (\partial / \partial \omega) f_T (r, \omega; x, y)| \geqslant T |x| / 2 \pi |y| - 2 |y| |\sinh r|, 
	\end{align*}
	and hence $| (\partial / \partial \omega) f_T (r, \omega; x, y)| \Gt T |x|/|y|$ provided that $|x| \Gt |y|^2 / M^{1-\sepsilon} T $. Applying Lemma \ref{lem: staionary phase, dim 1, 2}, with $Y = |x|$, $Q = 1$, $U = |y| / M^{1+\sepsilon}$ and $R = T|x|/|y|$, we deduce that the integral $  H_{T, \shskip  M}^{\ssstyle \sharp}   (  x   + iy ) $ in \eqref{3eq: H-sigma} is negligibly small unless 
	\begin{align}\label{3eq: |x| small}
	|x| \Lt  |y|^2 / M^{1-\sepsilon} T .
	\end{align}  
	Finally, note that $ |x| \Lt |y|^2 / M^{1-\sepsilon} T $ amounts to $ |x| \Lt |z|^2 / M^{1-\sepsilon} T $, and  in this case $|y| \asymp |z|$ so that \eqref{app: main bound for H(z)} follows from \eqref{app: bound for H sharp}.
\end{proof}

\begin{rem}\label{rem: HTM}
From the uniform asymptotic formulae in {\rm\cite{Balogh-K,Dunster-Bessel}} {\rm(}cf. also {\rm\cite{Olver-1,Olver}}{\rm)} for $K_{i \varnu} (    \varnu z)$ and $ K_{i \varnu} (  e^{\pm \pi i  } \varnu z) $ {\rm(}$\varnu > 0${\rm)}, one may derive the uniform asymptotic formula 
\begin{align*}
 \boldJ_{it} (Z/2\pi) \sim \frac { 2 \pi} {\sqrt{\left|t^2 + Z^2\right|}} \cos  (4 \pi t \, \Im \, \xi (  Z/ i t ) / \pi  ) 
\end{align*}
for $|\Re (Z/it)| >   c' >  \pi / 2$, where
\begin{align*}%\label{app: Liouville transform, 2}
\xi (  Z/ it ) %= \frac {2} {3} (- \zeta (  Z/ it ))^{3/2}   
=  i \log \bigg( \frac { t +  \sqrt{  t^2 + Z^2 } } {  Z } \bigg) - \frac {\pi } 2 - i  \sqrt { 1 + \frac {Z^2} {t^2} }.
\end{align*}
It may be shown that the same assertions in Lemma {\rm \ref{lem: HTM(z), 2}} {\rm(2)} are valid for  $|\Im  ( z)| > c T/ 2 \pi$ with $c > \pi / 2$ {\rm(}note the condition $c < 1$ in Lemma {\rm \ref{lem: HTM(z), 2}}{\rm)}. There is no room for any improvement {\rm(}though the $ \vepsilon$ in {\rm\eqref{app: main bound for H(z)}} is removable{\rm)}. This is because   there is no oscillation in the Bessel function $\boldJ_{it} (Z/2\pi)$ when $|\Re (Z)|$ is small{\rm;} for example, when $Z = 2\pi i y$ {\rm(}$y $ real and $|y| > t/4${\rm)}, we would have $$ \boldJ_{it} (i y) \sim \frac {2\pi} { \sqrt{  4\pi^2 y^2 - t^2} } . $$ 
 
By the proof of Lemma {\rm\ref{lem: HTM(z), 2}}, one may also show that	\begin{align}\label{3eq: H-natural (z), H-flat (z)}
  \frac {\partial H_{T, \shskip  M}^{  \ssstyle \sharp}   (  x e^{i\theta} )} {\partial \theta} \Lt M T^{2 + \sepsilon}, \quad 
	  \frac {\partial^2 H_{T, \shskip  M}^{  \ssstyle \sharp}   (  x e^{i\theta} )} {\partial x \partial \theta }    \Lt \frac {T^{3 + \sepsilon}}  { x } ,
	\end{align}
	and that they are negligibly small unless $|\cos \theta| \Lt x / M^{1-\sepsilon} T$. For    reasons similar to the above, these assertions can not be improved in an essential way. Note that one loses $x = |z|$ when taking the $\theta$-derivative and the bound $M T^{2+\sepsilon}$ is on the order of $H_{T, \shskip M}$ {\rm(}$\asymp M T^2${\rm)}. 
\end{rem}

For the proof of Theorem {\rm\ref{thm: t-aspect, 1}}, we only need Lemma \ref{lem: HTM(z), 2} (1), indeed, its simple corollary as below. 

\begin{cor}\label{cor: HTM(z), 2}
	We have  $H_{T, \shskip  M} (z) = O_{A} (T^{- A})$ for $1 < |z| \Lt T$. 
\end{cor}

\subsection{Estimates for $\widetilde{\mathrm{KB}}_{\fq} (\mathfrak n; h) $ and $\mathrm{KB}_{(1)} (\mathfrak n_1, \frn_2; h_{T, \shskip  M}) $} Recall the definitions of the Kloosterman--Bessel terms $\widetilde{\mathrm{KB}}_{\fq} (\mathfrak{n}; h) $, ${\mathrm{KB}}_{\fq} (\mathfrak n; h) $  and $\mathrm{KB}_{(1)} (\mathfrak n_1, \frn_2; h ) $ in \eqref{2eq: tilde KBq}, \eqref{2eq: KBq} and \eqref{2eq: KBq, 2}. 

\begin{lem} \label{lem: bounds for KB}
	We have 
\begin{equation}\label{3eq: bound for BKq} 
	\widetilde{\mathrm{KB}}_{\fq} (\mathfrak n; h) \Lt_{ \, h, \shskip F, \shskip \sepsilon} {\RN (\frn)^{\frac 1 2 - \sepsilon }} / { \RN (\fq)^{  \frac 1 2 - \sepsilon } }, 
\end{equation}
	and, for $ \RN (\frn_1 \frn_2) \Lt   T^4 $,
\begin{equation}\label{3eq: bound for BKTM}
\mathrm{KB}_{(1)} (\mathfrak n_1, \frn_2; h_{T, \shskip  M}) \Lt_{   \, h, \shskip F}  M   \RN (\frn_1 \frn_2)^{\frac 1 4 + \sepsilon } / T . 
\end{equation}
\end{lem}

\begin{proof}
We first prove \eqref{3eq: bound for BKq}  by appealing to the Weil bound \eqref{1eq: Weil} for Kloosterman sums and the estimates for $H (z)$ in Lemma \ref{lem: H(z), 1} and \ref{lem: H(z), 2.1}. For $\RN ( \frn) \leqslant \RN (\fq \frd)^2$, we have
\begin{align*}
\mathrm{KB}_{\fq} (\mathfrak n; h) \Lt \sum_{  c \shskip \in \fq \frd \smallsetminus \{0\} } \frac {\RN (c)^{\frac 1 2 + \sepsilon} } {\RN (c)} \frac {\RN (\frn)^{\frac 1 2 }} {\RN (c)} \Lt \frac {\RN (\frn)^{\frac 1 2 }} { \RN (\fq)^{\frac 3 2 - \sepsilon} } .
\end{align*}
For $\RN ( \frn) > \RN (\fq \frd)^2$, we have
\begin{align*}
\mathrm{KB}_{\fq} (\mathfrak n; h)   \Lt \hskip -1pt \mathop{\sum_{  c \shskip \in \fq \frd \smallsetminus \{0\} }}_{\RN (c) \geqslant \sqrt{ \RN (\frn)}  } \hskip -1pt
\frac {\RN (c)^{\frac 1 2 + \sepsilon} } {\RN (c)} \frac {\RN (\frn)^{\frac 1 2 }} {\RN (c)}
+ \hskip -1pt
\mathop{\sum_{  c \shskip \in \fq \frd \smallsetminus \{0\} }}_{\RN (c) <  \sqrt{ \RN (\frn)}  } \hskip -1pt
\frac {\RN (c)^{\frac 1 2 + \sepsilon} } {\RN (c)} \frac {\RN (c)^{\frac 1 2}} {\RN (\frn)^{\frac 1 4 }}  
\Lt  \frac {\RN (\frn)^{\frac 1 4 + \sepsilon }} { \RN (\fq) } . 
\end{align*}
Hence uniformly
\begin{equation*}%\label{3eq: bound for BKq, uniform}
\mathrm{KB}_{\fq} (\mathfrak n; h) \Lt_{ \, h, \shskip F, \shskip \sepsilon} {\RN (\frn)^{\frac 1 2 - \sepsilon }} / { \RN (\fq)^{ \frac 3 2 - \sepsilon } }.  
\end{equation*}
Then
\begin{align*}
\widetilde{\mathrm{KB}}_{\fq} (\mathfrak n; h) & = \sum_{ \fq = \frv \mathfrak{w}  } \mu (\mathfrak{w}) \RN (\frv) \sum_{\   \mathfrak{l} | \mathfrak{w}^{\infty} } \frac {\mathrm{KB}_{\fq} (\mathfrak l^2 \frn ; h)} {\RN (\mathfrak l)} \\
& \Lt \frac {\RN (\frn)^{\frac 1 2 - \sepsilon }} {\RN (\fq)^{\frac 3 2 - \sepsilon} } \sum_{ \fq = \frv \mathfrak{w}  }   \RN (\frv) \sum_{\   \mathfrak{l} | \mathfrak{w}^{\infty} } \frac {1} {\RN (\mathfrak l)^{\sepsilon}} \Lt \frac {\RN (\frn)^{\frac 1 2 - \sepsilon }} {\RN (\fq)^{\frac 1 2 - \sepsilon} } .
\end{align*}
To prove \eqref{3eq: bound for BKTM}, we use Lemma \ref{lem: HTM(z), 1} and Corollary \ref{cor: HTM(z), 2}. For $\RN ( \frn_1 \frn_2) \Lt T^4$, we have
\begin{align*}
 \mathrm{KB}_{(1)} (\mathfrak n_1 \frn_2; h_{T, \shskip  M}) & \Lt \frac {M} {T} \mathop{\sum_{  c \shskip \in \fq \frd \smallsetminus \{0\} }}_{ \RN (c) \geqslant \sqrt { \RN (\frn_1 \frn_2)}} \frac {\RN(\frn_1 , \frn_2 , c \frd\-)^{\frac 1 2} \RN (c)^{\frac 1 2 + \sepsilon} } { \RN (c)} \frac {\RN (\frn_1 \frn_2)^{\frac 1 2 }} {\RN (c)} + \frac 1 {T^A} \\
 & \Lt \frac { M \shskip \RN (\frn_1 \frn_2)^{\frac 1 4 + \sepsilon } } {T } . 
\end{align*}
\end{proof}

\begin{cor}
	Let $h (t),   h_{T, \shskip  M} (t) \in \mathscr{H}^+ (S, N)$ be as in Definition {\rm \ref{def: test functions}}.  	Then 
	\begin{align}\label{3eq: Weyl law, 1}
\Delta_{\fq} (1; h) =	\sum_{f \in H^{\star} (\fq)} \hskip -1pt \omega_f^{\star}  h  ( t_f )  \asymp  \varphi (\fq)  , 
	\end{align}
	and, when $\fq = (1)$,
	\begin{align}\label{3eq: Weyl law, 2}
\Delta_{(1)} (1; h_{T, \shskip  M}) =	\sum_{f \in H^{\star} (1)} \hskip -1pt \omega_f^{\star}    h_{T, \shskip  M}  ( t_f ) \sim \frac {16 M T^2 } {\pi^2 \sqrt{|d_F|}} \int_{-\infty}^{\infty} h (t) d t \asymp M T^2,
	\end{align}
	with the implied constants depending only on $h  $ and $F$. 
\end{cor}

\begin{proof}
	Let $\frn = (1)$ in \eqref{2eq: Kuznetsov, 2}. Then \eqref{3eq: Weyl law, 1} and \eqref{3eq: Weyl law, 2} are obvious in view of the estimates for  the Kloosterman--Bessel term as in \eqref{3eq: bound for BKq} and \eqref{3eq: bound for BKTM}. Note  that
	\begin{align*}
	H_{T, \shskip  M} = \int_{-\infty}^{\infty} h_{T, \shskip  M} (t) t^2 d t \sim 2 M T^2 \int_{-\infty}^{\infty} h  (t)   d t ,
	\end{align*}
	\begin{align*}
K_{T, \shskip  M} = \int_{-\infty}^{\infty} h_{T, \shskip  M} (t) \log^2 ( |t| + 3) d t \sim 2 M \log^2 T \int_{-\infty}^{\infty} h  (t)   d t .
	\end{align*}
\end{proof}

\section{The explicit formula}
Recall the notation  in \S \ref{section: newforms}. Let  $ f    \in  H^{\star} (\fq) $. Define its $1$-level density $D_1(f , \phi, R) $  as in Definition \ref{def: 1-level, 2-level density} (see also Remark \ref{rem: GRH}). %Assuming the Riemann hypothesis for $L(s, f  )$, we may write the non-trivial zeros of $L(s, f  )$ by $\rho_f= \frac{1}{2} + i \gamma_f$ with $\gamma_f$ real. 

Following from \cite[\S 4]{ILS-LLZ}, we have the explicit formula
\begin{align*}
	D_1(f , \phi, R) =  & \, \frac{\widehat{\phi}(0)}{\log R} \left( \log  \RN(\fq) + 2 \log |d_F|     -4 \log 2\pi \right) \\
	& + \sum_{ \pm } \frac{2}{\log R} \int_{-\infty}^{\infty} \frac{\Gamma'}{\Gamma} \bigg( \frac{1}{2} \pm i t_f + \frac{2\pi i x}{\log R}\bigg) \phi (x) d x   \\
	& - 2 \sum_{\fp} \sum_{\varnu=1}^{\infty} \widehat{\phi} \bigg( \frac{\varnu \log \RN(\fp)}{\log R}\bigg)
	\frac{a_f(\fp^{\varnu}) \log \RN(\fp)}{\RN(\fp)^{\varnu/2}\log R},
\end{align*}
where $a_f(\fp^{\varnu}) = \alphaup_f (\fp)^{\varnu} + \betaup_f (\fp)^{\varnu}$ for $\fp \hskip -1pt \nmid \hskip -1pt \fq$ and $a_f(\fp^{\varnu}) = \lambdaup_f^{\varnu} (\fp)$ for $\fp | \fq$.

\begin{lem}\label{lem: explicit formula} For $ f    \in  H^{\star} (\fq) $, we have
	\begin{equation}\label{2eq: D1 = ... -2P1-2P2}
	\begin{split}
	D_1 (f , \phi, R) = & \, \frac{\widehat{\phi} (0)}{\log R} \Big(\log \RN(\fq) + 2 \log \Big(\frac{1}{4}+ t_f^2 \Big) \Big)
	+  \frac{1}{2} \phi (0)   \\
& -   P_1(f , \phi, R) -   P_2(f , \phi, R)	+ O \left(  \frac{ \log \log 3 \RN(\fq) }{\log R} \right) , 
	\end{split}
	\end{equation}
	where 
	\begin{equation}\label{2eq: P1, P2}
P_{\varnu}(f , \phi, R) = 2   \sum_{\fp \shskip \nmid \fq}  \widehat{\phi} \left( \frac{\varnu \log \RN(\fp)}{\log R} \right) \frac{ \lambdaup_f(\fp^{\varnu}) \log \RN(\fp)}{\RN(\fp)^{\varnu/2} \log R}, \quad \varnu = 1, 2.
	\end{equation}
\end{lem}

\begin{proof}
	First, we have (cf. \cite[(4.14), (4.15)]{ILS-LLZ})
	\begin{align*}
\sum_{ \pm}	\frac{2}{\log R} \int_{-\infty}^{\infty} \frac{\Gamma'}{\Gamma} \hskip -1pt \left( \frac{1}{2} \pm i t_f + \frac{2\pi i x}{\log R}\right) \hskip -1pt \phi (x) d x 
	= \frac{2 \log \big(\frac{1}{4} + t_f^2 \big)}{\log R} \widehat{\phi}(0) + O \hskip -1pt \left(  \frac{1}{\log R}\right) \hskip -1pt.
	\end{align*}
For $\fp | \fq$, we have  $|a_f(\fp^{\varnu})|= \RN(\fp)^{- \frac 1 2{\varnu} }$ by \eqref{1eq: Hecke rel, 2} and \eqref{1eq: lambda p, p|q}.	
For $\fp \hskip -1pt \nmid \hskip -1pt \fq,$ in view of \eqref{Ramanujan}, we have $|\alphaup_f (\fp)|, |\betaup_f (\fp)| \le \RN(\fp)^{\frac{7}{64}}$ and hence $|a_f(\fp^{\varnu})| \le 2 \RN(\fp)^{\frac{7}{64} \varnu}$. Consequently, by trivial estimations,
\begin{equation*}
\sum_{\fp} \sum_{\varnu = 3}^{\infty} \widehat{\phi} \left( \frac{\varnu \log \RN(\fp)}{\log R}\right) \frac{a_f(\fp^{\varnu}) \log \RN(\fp)}{\RN(\fp)^{\varnu/2} \log R} \Lt \frac{1}{\log R},
\end{equation*} 
and
\begin{align*}
\sum_{\fp | \fq } \sum_{\varnu =1, \shskip 2}  \widehat{\phi} \left( \frac{\varnu \log \RN(\fp)}{\log R}\right) \frac{a_f(\fp^{\varnu}) \log \RN(\fp)}{\RN(\fp)^{\varnu/2} \log R}
\Lt \frac{\log \log 3 \RN(\fq)}{\log R}.
\end{align*}
Moreover,  the Landau Prime Ideal Theorem for $F$  implies
\begin{equation*}
  \sum_{\fp \shskip \nmid \fq} \widehat{\phi} \left( \frac{2 \log \RN(\fp)}{\log R} \right) \frac{2 \log \RN(\fp)}{\RN(\fp) \log R}
\ = \  \frac{1}{2} \phi (0) + O \left(  \frac{1}{\log R}\right).
\end{equation*}
Finally, \eqref{2eq: D1 = ... -2P1-2P2} follows from combining the computations above along with the relations $\alphaup_f (\fp) + \betaup_f (\fp)= \lambdaup_f(\fp)$ and $ \alphaup_f^2(\fp) + \betaup_f^2(\fp)= \lambdaup_f(\fp^2) -1$.
\end{proof}

%The  sum $P_{1}(f , \phi, R)$ is quite small compared to  $P_{2}(f , \phi, R)$

\section{Proof of the theorems}

\subsection{Proof of Theorem \ref{thm: level-aspect}} 

In view of Lemma \ref{lem: explicit formula} and Corollary \ref{cor: Delta (n) =}, we have
\begin{align*}
\SD_1 (H^{\star} (\fq), \phi; h) \hskip -0.5pt = \hskip -0.5pt \widehat {\phi} (0) \hskip -0.5pt + \hskip -0.5pt \frac 1 2 \phi (0) \hskip -0.5pt - \hskip -0.5pt (\SP_1 \hskip -0.5pt +  \hskip -0.5pt \SP_2) (H^{\star} (\fq), \phi; h) \hskip -0.5pt +  \hskip -0.5pt O \hskip -1pt \left(  \frac{ \log \log 3 \RN(\fq) }{\log \RN (\fq)} \right) \hskip -1pt,
\end{align*}
with 
\begin{align*}
\SP_{\varnu} (H^{\star} (\fq), \phi; h) \hskip -1pt & = \hskip -1pt \Avg_{\fq} \big(P_{\varnu}(f , \phi, \RN(\fq));   \omega^{\star}_f h (t_f) \big) \\
& = \hskip -1pt \sum_{\fp \shskip \nmid \fq}  \widehat{\phi} \left( \hskip -1pt \frac{ \varnu \log \hskip -1pt \RN(\fp)}{\log \RN (\fq)} \hskip -1pt \right) \hskip -1pt \frac{ 2 \log \RN(\fp)}{\RN(\fp)^{\varnu/ 2} \hskip -1pt \log \hskip -1pt \RN (\fq)}   \frac {  \widetilde{\mathrm{KB}}_{\fq} (  \fp^{\varnu}  ; h) + O \lp 1/ \varphi(\fq) \rp \hskip -1pt  } { \Delta_{\fq} (1; h) } .
\end{align*}
Also recall from \eqref{3eq: Weyl law, 1} that $ \Delta_{\fq} (1; h) \asymp \varphi (\fq) $.

Suppose that $\widehat \phi$ is supported on $[-\varv, \varv]$. We only consider the Kloosterman--Bessel contribution in the case $ \varnu = 1$. The other cases follow entirely analogously (in fact, we obtain better bounds). 
 By \eqref{3eq: bound for BKq},
\begin{align*}
  \sum_{\fp \shskip \nmid \fq}  \widehat{\phi} \left( \frac{  \log \RN(\fp)}{\log \RN (\fq)} \right)  \frac{ 2 \log \RN(\fp)}{\RN(\fp)^{\frac 1 2} \log \RN (\fq)}     \widetilde{\mathrm{KB}}_{\fq} (  \fp  ; h) & \Lt  \frac 1 {\RN (\fq)^{\frac 1 2 - \sepsilon}} \mathop{\sum_{\fp \shskip \nmid \fq}}_{\RN(\fp) \leqslant \RN (\fq)^{\varv} } \frac{  1 }{\RN(\fp)^{\sepsilon} } \\
  & \Lt   {\RN (\fq)^{\varv - \frac 1 2 + \sepsilon } }.
\end{align*}
Therefore this sum is $o (\varphi (\fq))$ as long as $\varv < \frac 3 2$, so that the Kloosterman--Bessel term in $\SP_{1} (H^{\star} (\fq), \phi; h)$ would have no
contribution to the main term of $ \SD_1 (H^{\star} (\fq), \phi; h) $.

\subsection{Proof of Theorem \ref{thm: t-aspect, 1}} \label{sec: Proof of Theorem 1.4}
By Lemma \ref{lem: explicit formula} and Corollary \ref{cor: Delta (n1, n2) =}, we have
\begin{align*}
\lim_{T \ra \infty} \SD_1 (H^{\star} (1), \phi; h_{T, \shskip  M})   =    \widehat {\phi} (0)   +   \frac 1 2 \phi (0)   -   \lim_{T \ra \infty} (\SP_1 + \SP_2) (H^{\star} (1), \phi ; h_{T, \shskip  M})   \hskip -1pt,
\end{align*}
with 
\begin{align*}
\SP_{\varnu} (H^{\star} & (1),   \phi; h_{T, \shskip  M})  = \Avg_{(1)} \big( P_{\varnu} (f , \phi, T^4); \omega^{\star}_f h_{T, \shskip  M} (t_f) \big) \\
& = \hskip -1pt \sum_{\fp  }  \widehat{\phi} \left( \hskip -1pt \frac{ \varnu \log \hskip -1pt \RN(\fp)}{4 \log T} \hskip -1pt \right) \hskip -1pt \frac{   \log \RN(\fp)}{2 \RN(\fp)^{\varnu/ 2} \hskip -1pt \log \hskip -1pt T}   \frac {   {\mathrm{KB}}_{(1)} (  \fp^{\varnu}, 1 ; h_{T, \shskip  M}) + O \lp M \log^2 T \rp \hskip -1pt  } { \Delta_{(1)} (1; h_{T, \shskip  M}) }. 
\end{align*}
For the term $\widehat{\phi} (0)$, we have used  here 
\begin{align*}
\Avg_{(1)} \big(  \log \big( \tfrac 1 4 + t_f^2 \big); \omega^{\star}_f h_{T, \shskip  M} (t_f) \big) \sim  \log T^2,
\end{align*}
as 
\begin{align*}
\int_{-\infty}^{\infty}  h_{T, \shskip  M} (t) \log \big( \tfrac 1 4 + t^2 \big) t^2 dt \sim 2 M T^2 \log T^2 \int_{-\infty}^{\infty} h (t) d t . 
\end{align*}

 Suppose that $\widehat \phi$ is supported on $[-\varv, \varv]$ for $\varv < 1$. Since $T^{4\varv} \Lt T^4$, the estimate \eqref{3eq: bound for BKTM} in   Lemma \ref{lem: bounds for KB} yields
\begin{align*}
 \sum_{\fp }   \widehat{\phi} \hskip -1pt \left( \hskip -1pt \frac{  \log \RN(\fp)}{4 \log T} \hskip -1pt \right) \hskip -1pt \frac{   \log \RN(\fp) \mathrm{KB}_{(1)} (\fp; h_{T, \shskip  M}) }{2 \RN(\fp)^{\frac 1 2} \hskip -1pt \log T}  
  \Lt \hskip -1pt \frac {M} {T } \hskip -1pt \sum_{ \RN (\fp) \leqslant T^{4 \varv} } \frac  {   \log \RN(\fp)} {\RN(\fp)^{\frac 1 4 - \sepsilon}}   \Lt \hskip -1pt M T^{3 \varv - 1 + \sepsilon }, 
\end{align*}
which is $o (MT^2)$   as desired. Recall here that the total mass $ \Delta_{(1)} (1; h_{T, \shskip  M}) $ is on the order of $MT^2$.  

\begin{rem}\label{rem: failure [-1, 1]}
When we extend the support of $\widehat{\phi}$ beyond the segment $[-1, 1]$, new  terms in  $ \SB_1  (\phi; h_{T, \shskip  M})$ are expected to contribute to the asymptotics. 
For this, we need to consider
\begin{align}\label{5eq: remark, sum}
  \frac {4} { \pi^2 |d_F|  } \mathop{ \sum_{\frd | (c)}}_{\RN (c) \Lt \shskip T^{2\varv - 2} } \frac {1} {\RN (c)} \sum_{ p } S (p, 1; c)   \frac{   \log \RN(p)}{  \sqrt{\RN(p)}  \hskip -1pt \log \hskip -1pt T}  H^{  \ssstyle \sharp }_{T, \shskip  M} \lp \frac {\sqrt{p}} {c} \rp \widehat{\phi} \left( \hskip -1pt \frac{   \log \hskip -1pt \RN(p)}{4 \log T} \hskip -1pt \right) \hskip -1pt,
\end{align} 
where  the $p$-sum is over prime integers in $\frO$, and $ H^{  \ssstyle \sharp }_{T, \shskip  M} (z) $ is defined in Lemma {\rm\ref{lem: HTM(z), 2.0}}. 
Assuming	the Riemann hypothesis for   Hecke-character $L$-functions over $F$, we have
\begin{align}
\label{5eq: Sums of Kloosterman, 2}
{\mathop{\sum_{ |p| \shskip \leqslant x  }}_{0 \shskip < \shskip \arg (p)   \leqslant \shskip \theta}} \hskip - 11 pt S(  p, 1 ; c) \frac {\log \RN(p)} {\sqrt{\RN(p)}}  = \frac {\varw_F   x \shskip \theta   } { \pi  } \frac {   \mu (c\frd\-)^2 } {   \varphi (c\frd\-)} + O \big( \RN (c) (  x \RN(c))^{\sepsilon} \big). 
\end{align}
It follows  that the $p$-sum in {\rm\eqref{5eq: remark, sum}}  is equal to
\begin{align*}
 \frac {\varw_F} {  \pi \log T} \hskip -1pt \int_0^{\infty} \int_0^{2\pi}  \hskip -1pt \left\{    { \hskip -1pt  x \shskip \theta   }   \frac {   \mu (c\frd\-)^2 } {   \varphi (c\frd\-)} + O \big( \RN (c) (  x \RN(c))^{\sepsilon} \big) \hskip -1pt \right\} & \\
\cdot \frac {\partial^2} {\partial \theta \partial x} H_{T, \shskip  M}^{\ssstyle \sharp}   \bigg( \hskip -1pt \frac {\hskip -1pt \sqrt {x e^{i\theta}}} {c} \bigg)   \widehat{\phi} \hskip -1pt \left( \hskip -1pt \frac{  \log x }{2 \log T} \hskip -1pt \right) d \theta d x   & .
\end{align*}
The main term of this integral will eventually contribute the desired
\begin{align*}
- \frac { 8 H_{T, \shskip  M} } {   \pi^2 \hskip -1pt \sqrt{|d_F|} }    \lp \int_{- \infty}^{\infty}  {\phi} (x) \frac {\sin 2\pi x} {2\pi x} d x - \frac 1 2  {\phi} (0) \hskip -1pt \rp.
\end{align*}
The error term is bounded by
\begin{align*}
\RN (c) (\RN(c) T)^{\sepsilon} \int_0^{T^{\varv} / \hskip -1pt \sqrt {\RN(c)}  }   \int_0^{2\pi} \bigg( x  \bigg|    \frac {\partial^2 H_{T, \shskip  M}^{  \ssstyle \sharp}   (  x e^{i\theta} )} {\partial x \partial \theta }    \bigg| + \bigg| \frac {\partial H_{T, \shskip  M}^{  \ssstyle \sharp}   (  x e^{i\theta} )} {\partial \theta}  \bigg|\bigg)   d \theta {d x} .
\end{align*}
However, the estimates {\rm\eqref{3eq: H-natural (z), H-flat (z)}} in Remark {\rm\ref{rem: HTM}} for these derivatives could only be used to bound the error-term contribution by $ T^{2\varv + 2 + \sepsilon} / M $. Unfortunately, this is too large if $\varv > 1$. 
\end{rem} 

\subsection{Proof of Theorem \ref{thm: t-aspect, 2}}

Since the sign of functional equation of $L (s, f)$ for $f$ in the family $  H^{\star} (1)$ is always $\epsilon_f = 1$, %(as we have chosen $\Gamma_0 (1) = \GL_2 (\frO)$ instead of $\SL_2 (\frO)$), 
the vanishing order of $L(s,f)$  is {\it even} at $s= \frac 1 2$. Thus by definition (see Definition \ref{def: 1-level, 2-level density}), we have
\begin{align}
D_2(f, \phi_1, \phi_2, R) = D_1(f, \phi_1, R) D_1 (f, \phi_2, R) -2 D_1 (f, \phi_1 \phi_2, R).
\end{align}
Before applying  Lemma \ref{lem: explicit formula} to 
\begin{align}\label{4eq: D natural}
\SD_2^{\scriptscriptstyle \natural} (H^{\star} (1), \phi_1, \phi_2; h_{T, \shskip  M}) = \Avg_{(1)} \big( D_1(f, \phi_1, T^4) D_1 (f, \phi_2, T^4) ; \omega^{\star}_f h_{T, \shskip  M} (t_f) \big) ,
\end{align}
we note (see \S \ref{sec: Proof of Theorem 1.4}) that  
\begin{align*}
\Avg_{(1)} \big(P_{\varnu} (f, \phi_{\mu}, T^4); \omega_f^{\star} h_{T, M}(t_f) \big)= o(1), \quad \varnu, \mu = 1, 2 ,
\end{align*}
for the support of $\widehat{\phi}_{\mu}$ contained in $(-1,1).$
We now apply  Lemma \ref{lem: explicit formula} and Corollary \ref{cor: Delta (n1, n2) =} to \eqref{4eq: D natural}. Some simple calculations show that  
\begin{align*}
  \lim_{T \ra \infty} \SD_2^{\scriptscriptstyle \natural} (H^{\star} (1) ,   \phi_1 , \phi_2 ; h_{T, \shskip  M}) =  &  \lp \widehat{\phi}_1(0) + \tfrac{1}{2} \phi_1(0) \rp \lp \widehat{\phi}_2 (0) + \tfrac{1}{2} \phi_2(0) \rp \\
& + \lim_{T \ra \infty} \big(\SP_1^{\shskip \mathrm{diag}}  + \SP_2^{\shskip \mathrm{diag}}  \big) (H^{\star} (1), \phi_1, \phi_2  ) \\
& + \lim_{T \ra \infty} \big(  \SP_{11}^{\shskip \mathrm{off}}  + \cdots + \SP_{22}^{\shskip \mathrm{off}} \big) (H^{\star} (1), \phi_1, \phi_2 ; h_{T, \shskip  M}) ,
\end{align*}
where 
\begin{align*}
\SP_\varnu^{\shskip \mathrm{diag}} (H^{\star} (1), \phi_1, \phi_2) =   \sum_{\fp}  \widehat{\phi}_1 \left( \frac{\varnu\log \RN(\fp)}{4\log T} \right)  \widehat{\phi}_2 \left( \frac{\varnu\log \RN(\fp)}{4\log T} \right) 
\frac{\log^2 \RN(\fp)}{4 \RN(\fp)^\varnu \log^2 T},
\end{align*}
and
\begin{align*}
\SP_{\varnu_1 \varnu_2}^{\shskip \mathrm{off}} (H^{\star} (1), & \, \phi_1 , \phi_2   ; h_{T, \shskip  M}) =   \mathop{\sum \sum}_{\fp_1, \, \fp_2}  \widehat{\phi}_1 \left( \frac{\varnu_1\log \RN(\fp_1)}{4\log T} \right)  \widehat{\phi}_2 \left( \frac{\varnu_2\log \RN(\fp_2)}{4\log T} \right) \\
&
\cdot \frac{\log  \RN(\fp_1) \log  \RN(\fp_2)}{4 \RN(\fp_1 )^{\varnu_1/2} \RN(\fp_2 )^{\varnu_2/2} \log^2 T} \frac {   {\mathrm{KB}}_{(1)} (  \fp_1^{\varnu_1}, \fp_2^{\varnu_2} ; h_{T, \shskip  M}) + O \lp M \log^2 T \rp \hskip -1pt  } { \Delta_{(1)} (1, 1; h_{T, \shskip  M}) }. 
\end{align*}
Suppose that the supports  of $\widehat{\phi}_{1}$ and $\widehat{\phi}_2$ are both contained in $\left(-\frac 1 2, \frac 1 2\right)$. Similar to the estimation in \S \ref{sec: Proof of Theorem 1.4}, one may prove that the off-diagonal $\SP_{\varnu_1 \varnu_2}^{\shskip \mathrm{off}}$ are all $o (1)$.  For the diagonal  $\SP_{\varnu }^{\shskip \mathrm{diag}}$,  the Prime Ideal Theorem for $F$ implies that
\begin{align*}
(\SP_1^{\shskip \mathrm{diag}} + \SP_2^{\shskip \mathrm{diag}}) (H^{\star} (1), \phi_1, \phi_2) \sim 2 	\int_{-\infty}^{\infty}  |y|  \widehat{\phi}_1(y)  \widehat{\phi}_2 (y) d y .
\end{align*}
Finally, by Theorem \ref{thm: t-aspect, 1},
\begin{align*}
\Avg_{(1)} \big(  D_1 (f, \phi_1 \phi_2, T^4);  \omega_f^{\star} h_{T, M}(t_f) \big) = \widehat{\phi_1 \phi}_2 (0) + \frac{1}{2} \phi_1(0) \phi_2(0).
\end{align*}
The proof is complete by combining the forgoing results.

\appendix 

\section{Low-lying zeros of $L$-functions for $\SL_2 (\BZ)$-Maass forms}\label{sec: app real}

In this appendix, we give a simple proof of (a variant of) the result in \cite{Alpoge-Miller-1} for the class of weight functions $h_{T, \shskip  M} (t)$ as in Definition \ref{def: test functions}. Moreover, we verify the Katz--Sarnak heuristic for even and odd Hecke--Maass forms.  Our analysis of the Bessel integrals is essentially due to Xiaoqing Li \cite{XLi2011}.

Let $ H^{\star} (1)  $ be the collection of (orthogonal) primitive Maass--Hecke cusp forms for $\SL_2 (\BZ)$. We shall also consider separately the subsets $H^{\ssstyle +} (1)$ and $H^{\ssstyle -} (1)$ of even and odd forms, respectively.  For each $f \in H^{\star} (1)  $, let  $    \frac 1 4 +  t_f^2$ be the Laplacian eigenvalue of $f $, let 
\begin{align}
\omega_f = \frac{1}{\cosh (\pi t_f)  \|f\|^2 } 
\end{align}
be the spectral weight of $f$ ($\|f\|$ is the $L^2$-norm of $f$ with respect to the Petersson inner product), and let $\lambdaup_f (n)$ be the Hecke eigenvalues of $f$. % Define 
%\begin{align*}
%\Avg  (A; w) = \frac{\sum_{f \in H^{\star} (1)} w (f) A(f )  }{\sum_{f \in H^{\star} (1)}   w (f) } .
%\end{align*}
For  $h_{T, \shskip  M} (t) $ as in Definition {\rm\ref{def: test functions}}, define the averaged $1$-level density (see Definition \ref{def: 1-level, 2-level density}):
\begin{align}
\SD_1 (H^{\sigmaup} (1), \phi; h_{T, \shskip  M}) = \frac{\sum_{f \in H^{\sigmaup} (1)} \omega_f h_{T, \shskip  M} (t_f) D_1 (f , \phi, T^2)  }{\sum_{f \in H^{\sigmaup} (1)}   \omega_f h_{T, \shskip  M} (t_f) }, % =  \Avg  \big( D (f , \phi, T^2); \omega_f h_{T, \shskip  M} (t_f) \big). 
\end{align}
where $\sigmaup = \star, +, -$.

The following theorem is a variant of \cite[Theorem 1.3]{Alpoge-Miller-1}.

\begin{thm}\label{thm: t-aspect, 1, Z}
	Let $T, M > 1$ be such that $M = T^{\mu}$ with $0 < \mu < 1$. Fix $h   \in \mathscr{H}^{\ssstyle +} (S, N)$ and define $h_{T, \shskip  M} $ by {\rm\eqref{0eq: defn of hTM}} in Definition {\rm\ref{def: test functions}}. Let  $\phi $ be an even Schwartz function  with the support of $\widehat{\phi}$ in $ (-1-\mu, 1+\mu )$. Assume the Riemann hypothesis {\rm(}for the Riemann zeta function{\rm)}. Then 
	\begin{align*}
	\lim_{T \ra \infty} 	\SD_1 (H^{\star} (1), \phi ; h_{T, \shskip  M}) = \int_{-\infty}^{\infty} \phi (x) W_1  (\mathrm{O}) (x) d x. 
	\end{align*}
\end{thm}

\begin{rem}
%	The Riemann hypothesis is assumed for the deduction of a nontrivial bound for the Eisenstein contribution as in {\rm \eqref{app: Xi (p)}}{\rm;} the trivial bound is insufficient if the support of $\widehat{\phi}$ is extended beyond $[-1, 1]$. On the other hand, 

When considering  the Eisenstein terms, the trivial bound is insufficient if the support of $\widehat{\phi}$ is extended beyond $[-1, 1]$. However, Alpoge and Miller {\rm\cite{Alpoge-Miller-1}} claim with no proof in their  {\rm (29)} a bound that seems too strong to be true and it leads to an estimate  $  T \log T$ for the total Eisenstein contribution in their {\rm(71)} {\rm(}saving a whole $T$ over the main term{\rm)}. An easy way to rectify this is to assume the Riemann hypothesis and use it to deduce a nontrivial bound for the Eisenstein terms as in {\rm \eqref{app: Xi (p)}} {\rm(}half as strong as their claimed {\rm (29)}{\rm)}. It seems to require some efforts to work without the Riemann hypothesis. 
\end{rem}

For $f \in H^{\ssstyle \pm} (1)$, its Fourier coefficients satisfy 
\begin{align}
\label{app: rho = lambda}
\rho_f (\pm n) = \pm   \lambdaup_f (  n), \quad \quad n > 0,
\end{align}  and $\epsilon_f = \pm 1$. 

A new result towards the Katz--Sarnak conjecture for the family $H^{\ssstyle \pm} (1)$ is as follows. 

\begin{thm}\label{thm: t-aspect, pm, Z}
	Let $\mu$,  $T$, $M$, $h_{T, \shskip  M}$ and $\phi$ be as in Theorem {\rm\ref{thm: t-aspect, 1, Z}}.  Assume the Riemann hypothesis for Dirichlet $L$-functions {\rm(}including the Riemann  zeta function{\rm)}. When $\mu > \frac 1 3$, we have
	\begin{align*}
	\lim_{T \ra \infty} 	\SD_1 (H^{\ssstyle +} (1), \phi ; h_{T, \shskip  M}) & = \int_{-\infty}^{\infty} \phi (x) W_1 (\mathrm{O(even)}) (x) d x, \\
	\lim_{T \ra \infty} 	\SD_1 (H^{\ssstyle -} (1), \phi ; h_{T, \shskip  M}) & = \int_{-\infty}^{\infty} \phi (x) W_1 (\mathrm{O(odd)}) (x) d x .
	\end{align*}
\end{thm}

\subsection{The Kuznetsov trace formula for $\SL_2 (\BZ)$}
Let  $h (t) \in \mathscr{H} (S, N) $ be as in Definition \ref{def: test functions}. For $n  > 0$, define
\begin{align}
\Delta^{\star}_{\ssstyle \pm} (n; h) = \sum_{f \in H^{\star} (1)} \omega_f h(t_f) \rho_f (n )  { \rho_f (\pm 1) } ,
\end{align}
\begin{align}
\Xi   (n ; h) =  \frac 1 { \pi} \int_{-\infty}^{\infty} \hskip -2pt  \omega  (t) h (   t ) 
\eta \hskip -1pt \lp  n , \hskip -1pt \tfrac 1 2 + i t \rp    \hskip -1pt  d t,
\end{align}
with
\begin{align}
\omega (t) = \frac {1 } {|\zeta (1 + 2it)|^2} , \quad \eta (n, s) =  \sum_{ \sstyle a | n  }    (n/a^2)^{ s - \frac 1 2}. 
\end{align}
The Kuznetsov trace formula for $\SL_2 (\BZ)$ is as follows (cf. \cite[\S 3]{CI-Cubic} or \cite{Kuznetsov}):
\begin{equation}\label{app: Kuznetsov}
\begin{split}
\Delta^{\star}_{\ssstyle \pm} (n ; h)  + \Xi   (n ; h) =  2 \delta_{   n , \, \scalebox{0.7}{$\pm$}  1} \shskip H  + 2 \mathrm{KB}^{\ssstyle \pm}  (n ; h) ,
\end{split}
\end{equation}
where
\begin{align}\label{app: KB +-}
\mathrm{KB}^{\ssstyle \pm} (n ; h) =  \sum_{  c = 1 }^{\infty}  \frac {S (  n, \pm 1 ; c)} { c}  H^{\ssstyle \pm} \bigg( \hskip -1pt   \frac {\textstyle \sqrt { n } } { c} \bigg),
\end{align} 
\begin{align}\label{app: integral H}
H   = \frac {1} {2 \pi^2} \int_{-\infty}^{\infty} h(t)    {  \tanh (\pi t) t \shskip dt }    , 
\end{align}
\begin{align} \label{app: H + (x)}
\quad \quad & H^{\ssstyle +} (x) \hskip -0.5 pt = \hskip -0.5 pt \frac { i} {\pi} \int_{-\infty}^{\infty} \hskip -1 pt h(t) J_{2it} (4\pi x) \frac {t \shskip dt } {\cosh (\pi t)},   \\
\label{app: H - (x)}
&  H^{\ssstyle -} (x) \hskip -0.5 pt = \hskip -0.5 pt \frac 2 {\pi^2  }   \int_{-\infty}^{\infty} \hskip -1 pt h(t) K_{2it} (4\pi x)   {\sinh (\pi t)} t \shskip dt , 
\end{align}
$J_{\varnu} (x)$ and $K_{\varnu} (x)$ are Bessel functions, $\delta_{ n , \, \scalebox{0.7}{$\pm$} 1} $ is the Kronecker $\delta$-symbol,  and $S (n,  \pm 1 ; c)$ is the Kloosterman sum.

By \eqref{app: rho = lambda}, one may restrict the spectral sum in \eqref{app: Kuznetsov} to even or odd forms in $H^{\ssstyle \pm} (1)$.  Precisely, for $n > 0$,  define
\begin{align}
\Delta^{\ssstyle \pm} (n; h) = \sum_{f \in H^{\scalebox{0.5}{$\pm$}} (1)} \omega_f h(t_f) \lambdaup_f (n ) ,
\end{align}
then
\begin{equation}\label{app: Kuznetsov, +-}
\begin{split}
\Delta^{\ssstyle \pm}  (n ; h)  + \delta_{\ssstyle \pm, \, +} \shskip \Xi   (n ; h) =    \delta_{   n , \shskip 1} \shskip H  +  \mathrm{KB}^{\ssstyle +}  (n ; h) \pm \mathrm{KB}^{\ssstyle -}  (n ; h) .
\end{split}
\end{equation}
%where $\sigmaup = +$, $-$.

\subsection{Analysis of the Bessel integral $H^{\ssstyle +}_{T, \shskip  M} (x)$}\label{sec: app H+(x)}
Let $  h_{T, \shskip  M} (t) \in \mathscr{H}^{\ssstyle +} (S, N)$ be as in Definition \ref{def: test functions}. Let    $H^{\ssstyle +}_{T, \shskip  M} (x)$  be the Bessel integral defined in \eqref{app: H + (x)} for $h_{T, \shskip  M} (t)$.   Set $M = T^{\mu}$ for  $0 < \mu < 1$.	 

\begin{lem}\label{lem: H(z), 1, real}
	We have  $ H^{\ssstyle +}_{T, \shskip  M} (x) \Lt_{\, h  }  M   x / T^{2  } $ for $x \leqslant 1$.
\end{lem}

\begin{proof}
	Let $ 3 < 2\delta < \min \left\{ 5, 2 S \right\} $.		By contour shift,  
	\begin{align*}
	\pi i	H^{\ssstyle +}_{T, \shskip  M}    (x)   =   \, &    h_{T, \shskip  M}  \hskip -1pt \lp - \tfrac 1 2 i \rp J_{1} (4 \pi x) - 3   h_{T, \shskip  M} \hskip -1pt \lp - \tfrac 3 2 i \rp J_{3} (4 \pi x) \\
	& - \int_{-\infty}^{\infty} h_{T, \shskip  M}  \hskip -1pt \lp   t - \delta i   \rp     {J_{  2it + 2 \delta } \lp 4 \pi  x \rp  }   \frac {   t - \delta i    } {\cos  (\pi  ( i t +  \delta   )  )} d t.
	\end{align*}
	By   Poisson's integral representation for $J_{\varnu} (x)$ (see \eqref{2eq: Poisson, J}), along with Stirling's formula,
	we infer that  
	\begin{align*}%\label{1eq: J for |z|<1}
	J_{1} (4 \pi x)  \Lt x , \hskip 10 pt J_{3} (4 \pi x)  \Lt x^3, \quad 
	\frac {J_{  2it + 2 \delta } \lp 4 \pi  x \rp  } {\cos  (\pi  ( i t +  \delta   )  ) }   \Lt    \lp \frac { x  } { |t| + 1   } \rp^{2 \delta}    , 
	\end{align*} 
	for $ x \leqslant 1$. 
	Consequently  \eqref{0eq: bounds for h(t)} and \eqref{0eq: defn of hTM} yield 
	\begin{align*}%\label{5eq: H(z) for z small}
	H^{\ssstyle +}_{T, \shskip  M} (z) \Lt e^{- T / M} (x + x^3) +   \frac { M   x^{2 \delta} } {T^{2 \delta -1}} \Lt  \frac { M   x  } {T^{2  }}  ,
	\end{align*} 
	for $x \leqslant 1$.
\end{proof}

\begin{lem}\label{lem: HTM(z), 2, real}
	When  $1 < x \Lt T  M^{1-\sepsilon} $, we have $H^{\ssstyle +}_{T, \shskip  M} (x) = O_{A} (T^{- A})$ for any $A > 0$.   
\end{lem}

\begin{proof} 
	First of all, by \cite[6.21 (12)]{Watson}, we may derive that
	\begin{align}
	{J_{2it } (4 \pi x) - J_{- 2it } (4 \pi x)}   = \frac 2 {\pi i} \sinh (\pi t)  \int_{-\infty}^{\infty} e (t r/\pi) \cos (4 \pi x \cosh r )  d r . 
	\end{align}
	This is a real analogue of \eqref{1eq: integral representation, 2}, but the integral here is {\it not} absolutely convergent. 
	When  $x      > 1$, it follows from partial integration  that only a negligibly small error will be lost if  the integral above is truncated at $|r| = T^{\sepsilon}$. Therefore, up to a negligible error, $H^{\ssstyle +}_{T, \shskip  M}   (  x   )$ is equal to
	\begin{align*}
	\frac 2 {\pi^2 } \int_{-\infty}^{\infty}  
	\int_{- T^{\ssepsilon}}^{T^{\ssepsilon}} \hskip -1 pt t \tanh (\pi t) h ((t-T)/M)   e (  t r/\pi) 
	\cos (4 \pi x \cosh r )      \shskip d r \shskip d\shskip t.
	\end{align*}
	By changing the variable from $t $ to $M t + T$, this integral turns into
	\begin{align*}
	\frac {2 M T} {\pi^2}   \int_{-\infty}^{\infty}    
	\int_{- T^{\ssepsilon}}^{T^{\ssepsilon}}   k (t) e ({   (Mt   +   T) r}/\pi)    \cos (4 \pi x \cosh r )  d r \shskip d\shskip t,
	\end{align*}
	with 
	\begin{align}\label{app: defn of k}
	k (t) = (Mt / T + 1)   \tanh (\pi (Mt   +   T)   ) h (t).  
	\end{align}
	By \eqref{0eq: bounds for h(t)}, it is easy to verify that
	\begin{align*}
	k (t+ i \sigma)  \Lt e^{- \pi |t|} (|t|+1)^{1-N},
	\end{align*}
	and hence $k (t)$ is a Schwartz function by Cauchy's integral formula. 
	After changing the order of integration, we see that the integral above is equal to
	\begin{align*}
	\frac {2MT} {\pi^2 }  
	\int_{- T^{\ssepsilon}}^{T^{\ssepsilon}} \widehat k \lp -   {   M r} / {\pi} \rp  e( Tr / \pi)   \cos (4 \pi x \cosh r ) d r. 
	\end{align*}
	Since $\widehat k (r)  $ is Schwartz, this integral may be effectively truncated at $|r| = M^{\sepsilon} / M$, and we need to consider  
	\begin{equation}\label{app: H+natural}
	H^{\ssstyle + \, \ssstyle \sharp }_{T, \shskip  M} (x) = \frac {2MT} {\pi^2 }  
	\int_{- M^{\ssepsilon} / M}^{M^{\ssepsilon}/ M} \widehat k \lp -   {   M r} / {\pi} \rp  e( Tr / \pi)   \cos (4 \pi x \cosh r ) d r. 
	\end{equation}
	We now apply Lemma \ref{lem: staionary phase, dim 1, 2} with the phase function
	$f_{T, \shskip \scalebox{0.7}{$\pm$} } (r) = T r/\pi \pm 2 x \cosh r. $
	Observe that $$\big|f_{T, \shskip \scalebox{0.7}{$\pm$} }' (r)\big| = |T / \pi \pm 2 x \sinh r|  \geqslant T/\pi - 2 x/ M^{1-\sepsilon} \Gt T $$ when $x \Lt T M^{1-\sepsilon}$. Lemma \ref{lem: staionary phase, dim 1, 2} implies that the integral is negligibly small for $x \Lt T M^{1-\sepsilon}$ (for applying Lemma \ref{lem: staionary phase, dim 1, 2} it would be more rigorous if we had used a smooth truncation at $|r| = M^{\sepsilon} / M$).
\end{proof}

Applying Lemma \ref{lem: H(z), 1, real}  and \ref{lem: HTM(z), 2, real}, along with the Weil bound for Kloosterman sums $S (  n, 1; c) \Lt c^{\frac 1 2 + \sepsilon} $, the Kloosterman--Bessel term $\mathrm{KB}^{\ssstyle +} (n, h_{T, \shskip  M})$ as in \eqref{app: KB +-}
may be estimated as follows,
\begin{align}\label{app: bound for KB}
\mathrm{KB}^{\ssstyle +} (n, h_{T, \shskip  M}) \Lt \frac {M} {T^{2  }} \sum_{ c \shskip \geqslant \sqrt n  } \frac {c^{\frac 1 2 + \sepsilon}} {c} \frac {n^{\frac 1 2 }} {c} + \frac 1 {T^A} \Lt \frac {M n^{\frac 1 4 + \sepsilon}} {T^{2   }} ,
\end{align}
for any ${n} \leqslant T^2 M^{2-\sepsilon}$. Moreover, in view of the lower bound $ |  \zeta  (1+2it) | \Gt 1/ \log (|t|+3)$ (cf. \cite[(3.11.10)]{Titchmarsh-Riemann}), the Eisenstein contribution $\Xi   (n ; h_{T, \shskip  M}) $ has the following bound,
\begin{align}\label{app: Eisenstein Xi}
\Xi   (n ; h_{T, \shskip  M}) = \frac 1 {   \pi} \int_{-\infty}^{\infty} \hskip -2pt   \frac { h_{T, \shskip  M} (   t ) 
	\eta \hskip -1pt \lp  n ,   \tfrac 1 2 + i t \rp } {|\zeta (1+2it)|^2}  d t = O (M \log^2 T) , \quad n = 1,   p^2.
\end{align} 
The trivial bound $O (M \log^2 T)$ above also holds for $ \Xi   (p ; h_{T, \shskip  M})  $, but this is not sufficient for our purpose. In order to get a better bound, we assume the Riemann hypothesis, under which we have
\begin{align}\label{app: Xi (p)}
\Xi   (p ; h_{T, \shskip  M}) =   O_{\delta, \shskip \sepsilon} \bigg(\frac {M T^{\sepsilon}} { p^{ \delta  }} \bigg),
\end{align} 
for any fixed $\delta < \frac 1 4$ and $\vepsilon > 0$. To prove this, we write 
\begin{align*}
\Xi   (p ; h_{T, \shskip  M}) = \frac 2 {   \pi} \int_{-\infty}^{\infty} \hskip -2pt   \frac { h_{T, \shskip  M} (t) p^{ i t}
} {\zeta (1+2it) \zeta (1-2it) }  d t ,
\end{align*}
and shift the integral contour to $\Im ( t) =  \delta < \frac 1 4$. We have uniformly
\begin{equation*}
\frac 1 {\zeta (\sigma + it)} = \left\{ \begin{split}
& O_{\delta, \shskip \sepsilon } ((|t|+3)^{\sepsilon}), & & \text{ if } \tfrac 1 2 < 1- 2 \delta \leqslant \sigma \leqslant 1, \\
& O (\log (|t|+3) ), & & \text{ if } \sigma \geqslant 1, 
\end{split} \right. 
\end{equation*}
where the first bound is valid under the Riemann hypothesis (cf. \cite[(3.11.8), (14.2.6)]{Titchmarsh-Riemann}). It follows that 
\begin{align*}
\Xi   (p ; h_{T, \shskip  M}) = \frac 2 {   \pi} \int_{-\infty}^{\infty} \hskip -2pt   \frac { h_{T, \shskip  M} (t + i \delta ) p^{ i t - \delta }
} {\zeta (1-2\delta +2it) \zeta (1+2\delta-2it) }  d t  =  O_{\delta, \shskip \sepsilon} \bigg(\frac {M T^{\sepsilon}} { p^{ \delta  }} \bigg),
\end{align*}
as desired.

As a consequence of \eqref{app: bound for KB} and \eqref{app: Eisenstein Xi} (with $n=1$), in view of \eqref{app: Kuznetsov}, the total mass of the average 
\begin{align}
\Delta^{\star}_{\ssstyle +} (1; h_{T, \shskip  M}) = \sum_{f \in H^{\star} (1)}   \omega_f    h_{T, \shskip  M}  ( t_f ) \sim \frac {2 M T } {\pi^2} \int_{-\infty}^{\infty} h (t) d t \asymp M T . 
\end{align}

\subsection{Proof of Theorem \ref{thm: t-aspect, 1, Z}} By standard arguments  as in  \cite{Alpoge-Miller-1} using the explicit formula and the Kuznetsov trace formula \eqref{app: Kuznetsov}, one is reduced to bounding certain sums involving the Eisenstein terms $\Xi$ and the Kloosterman--Bessel terms $\mathrm{KB}^{\ssstyle +}$ by   quantities which are of order strictly lower than $M T$. Two typical sums are
\begin{align*}
\SE_1 (\phi; h_{T, \shskip  M}) = \sum_{ p }   \widehat{\phi} \hskip -1pt \left( \hskip -1pt \frac{  \log p }{2 \log T} \hskip -1pt \right) \hskip -1pt \frac{   \log p  }{ \sqrt p    \log T} \Xi  (p; h_{T, \shskip  M}),
\end{align*}
and
\begin{align*}
\SB^{\shskip \ssstyle +}_1 (\phi; h_{T, \shskip  M}) = \sum_{ p }   \widehat{\phi} \hskip -1pt \left( \hskip -1pt \frac{  \log p }{2 \log T} \hskip -1pt \right) \hskip -1pt \frac{   \log p  }{ \sqrt p    \log T} \mathrm{KB}^{\ssstyle +}  (p; h_{T, \shskip  M}).
\end{align*}
Suppose that $\widehat \phi$ is supported on $ [-\varv, \varv] $ for $\varv < 1 + \mu$. 

For the first sum $\SE_1 (\phi; h_{T, \shskip  M})$, it follows from \eqref{app: Xi (p)} that
\begin{align*}
\sum_{ p }   \widehat{\phi} \hskip -1pt \left( \hskip -1pt \frac{  \log p }{2 \log T} \hskip -1pt \right) \hskip -1pt \frac{   \log p  }{ \sqrt p    \log T} \Xi  (p; h_{T, \shskip  M}) \Lt M T^{\sepsilon} \sum_{ p \shskip \leqslant T^{2 \varv} } \frac  {   \log p} { p^{\frac 1 2+ \delta }} \Lt M T^{ \lp 1 - 2 \delta \rp \varv + \sepsilon} ,
\end{align*}
and $M T^{ \lp 1 - 2 \delta \rp \varv + \sepsilon} = o (MT)$ if we choose $\delta$ so that $1 - 1/\varv < 2 \delta < \frac 1 2$. 

%\begin{rem}\label{rem: error in AM}
%	The analysis of the Eisenstein contribution is omitted by Alpoge and Miller {\rm\cite{Alpoge-Miller-1}}, while their {\rm (29)} seems to suggest {\rm(}in our setting{\rm)} the following strong but erroneous bound
%	\begin{align*}
%	\Xi   (p ; h_{T, \shskip  M}) = O   \bigg(\frac {M \log T} { \sqrt {p}} \bigg). 
%	\end{align*}
%\end{rem}

As for the second sum $\SB^{\shskip \ssstyle +}_1 (\phi; h_{T, \shskip  M})$, note that $T^{2\varv} \leqslant T^2 M^{2-\sepsilon}$. Hence \eqref{app: bound for KB} yields
\begin{align*}
\sum_{ p }   \widehat{\phi} \hskip -1pt \left( \hskip -1pt \frac{  \log p }{2 \log T} \hskip -1pt \right) \hskip -1pt \frac{   \log p  }{ \sqrt p   \log T} \mathrm{KB}^{\ssstyle +}  (p; h_{T, \shskip  M}) \Lt \frac {M} {T^2} \sum_{ p \shskip \leqslant T^{2 \varv} } \frac  {   \log p} { p^{\frac 1 4 - \sepsilon}} \Lt M T^{\frac 3 2 \varv - 2 + \sepsilon },
\end{align*}
which is $o (MT) $ as desired.

\subsection{Analysis of the Bessel integral $H^{\ssstyle -}_{T, \shskip  M} (x)$}\label{sec: app H-(x)}
Let    $H^{\ssstyle -}_{T, \shskip  M} (x)$  be the Bessel integral defined in \eqref{app: H - (x)} for $h_{T, \shskip  M} (t)$. Let $M = T^{\mu}$  with  $0 < \mu < 1$.	

\begin{lem}\label{lem: H-(x), 1}
	We have  $ H^{\ssstyle -}_{T, \shskip  M} (x) \Lt_{\, h  }  M   x / T^{2  } $ for $x \leqslant 1$.
\end{lem} 

\begin{proof}
	Recall the connection formula (\cite[3.7 (6)]{Watson}):
	\begin{align*}
	2 K_{\varnu} (z) = \pi \frac {
		I_{-\varnu} (z) - I_{\varnu} (z)} {\sin (\pi \varnu)} ,
	\end{align*}
	where $I_{\varnu} (z)$ is the $I$-Bessel function. So   \eqref{app: H - (x)}  may be rewritten as
	\begin{align*}
	H^{\ssstyle - }   (x) \hskip -0.5 pt = \hskip -0.5 pt  \frac i {\pi} \int_{-\infty}^{\infty} \hskip -1 pt h(t) I_{2it} (4\pi x) \frac {t \shskip dt } {\cosh (\pi t)} . 
	\end{align*}
	Note that $I_{\varnu} (z) = e^{-\frac 1 2 \pi i \varnu} J_{\varnu} (i z)$ (\cite[3.7]{Watson}). We may then follow the same line of arguments in the proof of Lemma \ref{lem: H(z), 1, real}. 
\end{proof}

\begin{lem}\label{lem: H-(x), 2}
	For $x > 1$, we have $ H^{\ssstyle -}_{T, \shskip  M} (x) = H^{\ssstyle - \, \ssstyle \sharp }_{T, \shskip  M} (x) + O_A  (T^{-A} ) $ with 
	\begin{align}\label{app: H-sharp}
	H^{\ssstyle - \, \ssstyle \sharp }_{T, \shskip  M} (x) = \frac {2MT} {\pi^2 }  
	\int_{- M^{\ssepsilon} / M}^{M^{\ssepsilon}/ M} \widehat k \lp -   {   M r} / {\pi} \rp  e( Tr / \pi)   \cos (4 \pi x \sinh r ) d r,
	\end{align}
	where $k (t)$ is a Schwartz function defined by {\rm\eqref{app: defn of k}}. 
\end{lem}

\begin{proof}
	Starting with the formula (\cite[6.23 (13)]{Watson}) 
	\begin{align}
	K_{2it} (4\pi x) = \frac 1 {2 \cosh (\pi t)} \int_{-\infty}^{\infty} e (t r/\pi) \cos (4 \pi x \sinh r )  d r ,
	\end{align}
	we may prove this lemma by the arguments that led us to \eqref{app: H+natural}. 
\end{proof}

Define $H^{\ssstyle - \, \ssstyle \natural }_{T, \shskip  M} (x) =   {d H^{\ssstyle - \, \ssstyle \sharp }_{T, \shskip  M} (x)} / {d x} $. By \eqref{app: H-sharp}, we have
\begin{align}\label{app: H-natural}
H^{\ssstyle - \, \ssstyle \natural }_{T, \shskip  M} (x) \hskip -0.5 pt   =   \hskip -0.5 pt - \frac {8 MT} {\pi}   \hskip -1 pt
\int_{- M^{\ssepsilon} / M}^{M^{\ssepsilon}/ M} \hskip -1 pt \widehat k \lp -   {   M r} / {\pi} \rp \sinh r \cdot  e( Tr / \pi)   \sin (4 \pi x \sinh r ) d r. 
\end{align}

\begin{lem}\label{lem: H-(x), 2.2}
	For $x > 1$, we have
	\begin{align}\label{app: bounds for H-}
	H^{\ssstyle - \, \ssstyle \sharp }_{T, \shskip  M} (x) \Lt T^{1 + \sepsilon}, \quad H^{\ssstyle - \, \ssstyle \natural }_{T, \shskip  M} (x) \Lt T^{1 + \sepsilon} / M.
	\end{align}
	Moreover, if $\mu > \frac 1 3$,  then $ H^{\ssstyle - \, \ssstyle \sharp }_{T, \shskip  M} (x), \, H^{\ssstyle - \, \ssstyle \natural }_{T, \shskip  M} (x) = O_A (T^{-A}) $ unless $ | T - 2 \pi x | \leqslant M^{1+ \sepsilon} $. 
\end{lem}

\begin{proof}
	Trivial estimation yields \eqref{app: bounds for H-}. As for the second statement, we apply Lemma \ref{lem: staionary phase, dim 1, 2}  with the phase function
	$f_{T, \shskip \scalebox{0.7}{$\pm$} } (r) = T r/\pi \pm 2 x \sinh r. $
	Observe that $$f_{T, \shskip \scalebox{0.7}{$+$} }' (r) =  T  /\pi + 2 x \cosh r > T/\pi, $$
	$$ f_{T, \shskip \scalebox{0.7}{$-$} }' (r) = T  /\pi - 2 x \cosh r =  T / \pi - 2 x  + O (xM^{\sepsilon} / M^2) ,$$
	and hence $\big| f_{T, \shskip \scalebox{0.7}{$-$} }' (r)\big| \Gt M^{1+\sepsilon} $ provided that $T < M^3$ and $ |T - 2\pi x| > M^{1+\sepsilon} $. 
\end{proof}

Similar to \eqref{app: bound for KB}, we may deduce from Lemma \ref{lem: H-(x), 1}, \ref{lem: H-(x), 2} and \ref{lem: H-(x), 2.2} that
\begin{align}\label{app: bound for KB-}
\mathrm{KB}^{\ssstyle -} (n, h_{T, \shskip  M}) \Lt \frac {M n^{\frac 1 4 + \sepsilon}} {T^{2   }}  ,
\end{align}
if $   n \Lt T^2 $. Consequently,  by \eqref{app: Kuznetsov, +-} and the estimates in \eqref{app: bound for KB}, \eqref{app: Eisenstein Xi} and \eqref{app: bound for KB-}, we infer that
\begin{align}
\Delta^{\ssstyle \pm} (1; h_{T, \shskip  M}) = \sum_{f \in H^{\scalebox{0.5}{$\pm$}} (1)} \omega_f h_{T, \shskip  M} (t_f) \sim \frac { H_{T, \shskip  M} } {2 \pi^2}  \asymp M T .
\end{align}

\subsection{Proof of Theorem \ref{thm: t-aspect, pm, Z}} \label{sec: even and odd}

In view of the Kuznetsov trace formula for $H^{\ssstyle \pm} (1)$ as in \eqref{app: Kuznetsov, +-}, we now also need to consider the contribution from the Kloosterman--Bessel terms $ \mathrm{KB}^{\ssstyle -} $. To be precise, we are left to consider 
\begin{align}\label{app: P1}
\SB^{\shskip \ssstyle -}_1 (\phi; h_{T, \shskip  M}) & =  \sum_{ p }   \widehat{\phi} \hskip -1pt \left( \hskip -1pt \frac{  \log p }{2 \log T} \hskip -1pt \right) \hskip -1pt \frac{   \log p  }{ \sqrt p    \log T} \mathrm{KB}^{\ssstyle -}  (p; h_{T, \shskip  M}), \\
\label{app: P2} \SB^{\shskip \ssstyle -}_2 (\phi; h_{T, \shskip  M}) & = \sum_{ p }   \widehat{\phi} \hskip -1pt \left( \hskip -1pt \frac{  \log p }{  \log T} \hskip -1pt \right) \hskip -1pt \frac{   \log p  }{   p    \log T} \mathrm{KB}^{\ssstyle -}  (p^2 ; h_{T, \shskip  M}) .
\end{align}

Suppose that $\widehat \phi$ is supported on $ [-\varv, \varv] $ for $\varv < 1 + \mu$. Set $P = T^{2 \varv}$. 

By \eqref{app: bound for KB-}, it is readily seen that $\SB^{\shskip \ssstyle -}_2 (\phi; h_{T, \shskip  M}) = O \big(M P^{\frac 1 4 + \sepsilon} / T^{ 2 } \big)$ and hence only contributes an error term. 

Since the support of $\widehat{\phi}$ is extended beyond the segment $[-1, 1]$, the sum $\SB^{\shskip \ssstyle -}_1 (\phi; h_{T, \shskip  M})$ would contribute to the asymptotics. For Theorem \ref{thm: t-aspect, pm, Z}, sufficient is to prove that it has the following asymptotic. 

\begin{lem}\label{lem: main term from Kloosterman}
	Assume  the Riemann hypothesis for the classical Dirichlet $L$-functions. Then, for $M = T^{\mu}$ with $\frac 1 3 < \mu < 1$, we have
	\begin{align}\label{app: P1 = W(O) - W(SO(even))}
	\SB^{\shskip \ssstyle -}_1 (\phi; h_{T, \shskip  M}) = - \frac {H_{T, \shskip  M}} {2\pi^2} \lp \int_{- \infty}^{\infty}  {\phi} (x) \frac {\sin 2\pi x} {2\pi x} d x - \frac 1 2  {\phi} (0) + O \lp \frac 1 {\log T} \rp \rp, 
	\end{align}
	where $ H_{T, \shskip  M} $ is the integral defined in {\rm\eqref{app: integral H}} for $h_{T, \shskip  M}$. 
\end{lem}

%For an evaluation of the sum of Kloosterman sums we need uniform asymptotics for primes in arithmetic progressions, and for the latter we apply the Riemann hypothesis for the classical Dirichlet $L$-functions.
%The rest of this section is devoted to the proof of Lemma \ref{lem: main term from Kloosterman}. 
\begin{proof} By Lemma \ref{lem: H-(x), 1}, \ref{lem: H-(x), 2} and \ref{lem: H-(x), 2.2}, we write
	\begin{align}\label{app: P1 (phi)}
	\SB^{\shskip \ssstyle -}_1 (\phi; h_{T, \shskip  M}) = \sum_{c \shskip \Lt \sqrt{P} / T } \frac { Q^{\ssstyle \sharp}_{T, \shskip  M} (c; \phi) } {c} + O \bigg(  \frac {M P^{\frac 3 4 + \sepsilon}} {T^2} \bigg), 
	\end{align}
	with
	\begin{align}\label{app: defn of Q(c)}
	Q^{\ssstyle \sharp}_{T, \shskip  M} (c; \phi) =   \sum_{ p }    S (p, -1; c)    \frac{   \log p  }{ \sqrt p  \hskip -1pt \log T}  H_{T, \shskip  M}^{\ssstyle - \, \sharp} \hskip -1pt \lp \frac {\hskip -1pt \sqrt p} {c} \rp \widehat{\phi} \hskip -1pt \left( \hskip -1pt \frac{  \log p }{2 \log T} \hskip -1pt \right) \hskip -1pt. 
	\end{align}
	
	Under the  Riemann hypothesis for classical Dirichlet $L$-functions (including the Riemann zeta function), it follows from Lemma 6.1 in \cite{ILS-LLZ} (see the paragraph above (6.5) in \cite{ILS-LLZ}) that 
	\begin{align}\label{app: Kloosterman over primes}
	{\sum_{ p \shskip \leqslant \shskip x }}  S (p, -1; c) \frac{\log p}{\sqrt{p}} =   \frac {2 \mu (c)^2 } {\varphi (c) } \hskip -1pt \sqrt{x}  + O \big( c (cx)^{\sepsilon} \big) .
	\end{align} 
	By \eqref{app: defn of Q(c)} and \eqref{app: Kloosterman over primes}, we get
	\begin{align*}
	Q^{\ssstyle \sharp}_{T, \shskip  M} (c; \phi) \hskip -1pt = \hskip -1pt - \frac 1 {\log T} \hskip -1pt \int_0^{\infty} \hskip -1pt \left\{ \hskip -1pt \frac {2\mu (c)^2  } {\varphi (c) } \hskip -1pt \sqrt{x}  + O \big( c (c x)^{\sepsilon} \big) \hskip -1pt \right\}  d  H_{T, \shskip  M}^{\ssstyle - \, \sharp} \hskip -1pt \lp \hskip -1pt \frac {\hskip -1pt \sqrt x} {c} \rp \hskip -1pt \widehat{\phi} \hskip -1pt \left( \hskip -1pt \frac{  \log x }{2 \log T} \hskip -1pt \right) \hskip -1pt .
	\end{align*}
	%By partial integration and a change of variable,  t
	The main term is equal to
	\begin{align*}
	%V^{\ssstyle \sharp}_{T, \shskip  M} (c; \phi)  = 
	\frac {2 \mu (c)^2 c} {\log T \varphi (c) }  \int_0^{\infty} H_{T, \shskip  M}^{\ssstyle - \, \sharp} (x) \widehat{\phi} \hskip -1pt \left( \hskip -1pt \frac{  \log (c x) }{  \log T} \hskip -1pt \right) \hskip -1pt    d x .
	\end{align*}
	The error term is bounded by 
	\begin{align*}
	c (c P)^{\sepsilon} \int_0^P \bigg( \frac {\sqrt x} c \bigg| H_{T, \shskip  M}^{\ssstyle - \, \natural} \bigg(\frac {\sqrt x} c \bigg)  \bigg| + \bigg| H_{T, \shskip  M}^{\ssstyle - \, \sharp} \bigg(\frac {\sqrt x} c \bigg)  \bigg|\bigg) \frac {d x} x & \\
	= 2 c (c P)^{\sepsilon} \int_0^{\sqrt P / c} \big(  \big|  H_{T, \shskip  M}^{\ssstyle - \, \natural} (x) \big| + \big|  H_{T, \shskip  M}^{\ssstyle - \, \sharp} (x) \big| x\- \big) d x & ,
	\end{align*}
	where $ H_{T, \shskip  M}^{\ssstyle - \, \sharp} (x) $ and $ H_{T, \shskip  M}^{\ssstyle - \, \natural} (x) $ are defined by \eqref{app: H-sharp} and \eqref{app: H-natural}. In view of Lemma \ref{lem: H-(x), 2.2}, one easily sees that it is bounded by $ T^{1+\sepsilon} c$ ($  H_{T, \shskip  M}^{\ssstyle - \, \sharp} (x)$ and $H_{T, \shskip  M}^{\ssstyle - \, \natural} (x)$ are essentially supported near $T/2\pi$), and that it only contributes an $O (\sqrt{P} T^{\sepsilon}) = o (MT)$ to $\SB^{\shskip \ssstyle -}_1 (\phi; h_{T, \shskip  M})$ as in \eqref{app: P1 (phi)}. We now change the $ H_{T, \shskip  M}^{\ssstyle - \, \sharp} (x) $ in the main term to  $ H_{T, \shskip  M}^{\ssstyle -} (x) $ and then complete the $c$-sum. By doing so, in view of Lemma \ref{lem: H-(x), 1} and \ref{lem: H-(x), 2}, we produce some extra terms which can be absorbed into the error term in \eqref{app: P1 (phi)}. By the definition of $H_{T, \shskip  M}^{\ssstyle -} (x)$ as in \eqref{app: H - (x)}, we write 
	\begin{align}\label{app: P1 = int K(t; phi)}
	\SB^{\shskip \ssstyle -}_1 (\phi; h_{T, \shskip  M}) =  \int_{-\infty}^{\infty} \hskip -1 pt h_{T, \shskip  M}(t)  K_{T} (t; \phi) {\sinh (\pi t)} t \shskip dt  + O \bigg(  \frac {M P^{\frac 3 4 + \sepsilon}} {T^2} \bigg),
	\end{align}
	with
	\begin{align}\label{app: K (t; phi)}
	K_{T} (t; \phi) =  \frac 4 {\pi^2  } \sum_{  c = 1 }^{\infty} \frac {\mu (c)^2} {\varphi (c) } \int_0^{\infty} K_{2it} (4\pi x)  \widehat{\phi} \hskip -1pt \left( \hskip -1pt \frac{  \log (c x) }{  \log T} \hskip -1pt \right) \hskip -1pt   \frac {d x} {\log T} .
	\end{align}
	Note that we need $|t|$ to be very close to $T$. We now follow the  arguments in \S 7 of \cite{ILS-LLZ}.
	By the definition \eqref{0eq: Fourier} and the formula (\cite[13.21 (8)]{Watson})
	\begin{align*}
	\int_0^{\infty} K_{\varnu} (x) x^{\shskip \mu-1} d x = 2^{\mu-2} \Gamma \lp \frac {\mu-\varnu} 2 \rp \Gamma \lp \frac {\mu+\varnu} 2 \rp, \quad \Re (\mu) > |\Re(\varnu)|, 
	\end{align*}
	along with Euler's reflection formula, we find that the integral in \eqref{app: K (t; phi)} is equal to 
	\begin{align*}
	\frac {1} {4 } \int_{-\infty}^{\infty} {\phi} (2 y \log T)    \frac {\lp   {2\pi} / c \rp^{ 4\pi i y  }} {\cosh (\pi t + 2\pi^2 y)} \frac { \Gamma \lp \frac 1 2  + it - 2 \pi i y  \rp } {\Gamma \lp \frac 1 2  + it +  2 \pi i y   \rp }         d y . 
	\end{align*}
	Next we interchange the integration over $y$ with the
	summation over $c$. For the convergence we introduce   $\vepsilon > 0$ so that
	\begin{align*}
	K_{T} (t; \phi) =  \frac 1 {\pi^2 } \lim_{\sepsilon \ra 0} \int_{-\infty}^{\infty} {\phi} (2 y \log T)   \frac {\chiup (\vepsilon + 4\pi i y) (    {2\pi}  )^{ 4\pi i y  }} {\cosh ( \pi t + 2\pi^2 y)} \frac { \Gamma \lp \frac 1 2 + it - 2 \pi i y  \rp } {\Gamma \lp \frac 1 2  + it +  2 \pi i y \rp  }         d y ,
	\end{align*}
	where 
	\begin{align*}
	\chiup (s) = \sum_{  c = 1 }^{\infty} \frac {\mu(c)^2} {c^s \varphi (c)} .
	\end{align*}
	We need $\chiup (s)$ for $s \Lt (\log T)\-$.  By comparing their Euler products, we find that the Laurent expansions of $\chiup (s)$ and $\zeta (1+s)$ near $s = 0$ have the same leading term $s^{-1}$. Thus  
	\begin{align*}
	\chiup (s) = s^{-1} + O( 1) .
	\end{align*}
	Moreover, by Stirling's formula 
	\begin{align*}
	\Gamma \lp \frac 1 2 + it + 2 \pi i y  \rp = \Gamma \lp \frac 1 2 + it - 2 \pi i y  \rp (it)^{4\pi i y} \lp 1 + O \lp    \frac {|y|} T   \rp \rp .
	\end{align*} 
	Hence we deduce that
	\begin{align*}
	K_T (t; \phi) \hskip -1pt = \hskip -1pt \frac 1 {\pi^2 \cosh (\pi t) } \bigg\{ \hskip -1pt \lim_{\sepsilon \ra  0} \int_{-\infty}^{\infty} {\phi} (2 y \log T) \hskip -1pt  \lp \frac {|t| } {2\pi} \rp^{\hskip -1pt 4\pi i y}   \hskip -1pt \frac  { d y } {\vepsilon \hskip -1pt - \hskip -1pt 4\pi i y} \hskip -1pt + \hskip -1pt O \hskip -1pt \lp \hskip -1pt \frac 1 {\log T} \hskip -1pt \rp \hskip -2pt \bigg\}  .
	\end{align*}
	The last  integral becomes 
	\begin{align*}
	&- \int_{-\infty}^{\infty} {\phi} (2 y \log T) \sin (4\pi y \log (|t|/2\pi)) \frac {d y} {4\pi y + i \vepsilon} \\
	& + \int_{-\infty}^{\infty} {\phi} (2 \vepsilon y \log T) \cos (4\pi \vepsilon y \log (|t|/2\pi)) \frac {d y} {(4\pi y)^2 + 1}.
	\end{align*}
	Now we can take the limit as $\vepsilon \ra 0$, getting
	\begin{align*}
	- \frac 1 2 \int_{-\infty}^{\infty} {\phi} (y)  \frac {\sin 2\pi y} {2\pi y}    {d y}  + \frac 1 4 \phi (0) + O \hskip -1pt \lp \hskip -1pt \frac 1 {\log T} \hskip -1pt \rp \hskip -1pt ,
	\end{align*}
	for $|t|$ very close to $T$. Thus
	\begin{align*}
	K_T (t; \phi) \hskip -1pt = \hskip -1pt - \frac 1 {2 \pi^2 \cosh (\pi t) } \bigg\{  \int_{-\infty}^{\infty} {\phi} (y)  \frac {\sin 2\pi y} {2\pi y}    {d y}  - \frac 1 2 \phi (0) + O \hskip -1pt \lp \hskip -1pt \frac 1 {\log T} \hskip -1pt \rp   \hskip -2pt \bigg\}  .
	\end{align*}
	We arrive at \eqref{app: P1 = W(O) - W(SO(even))} by introducing this into \eqref{app: P1 = int K(t; phi)}. 
\end{proof}

%	\bibliographystyle{alphanum}
	%    Insert the bibliography data here.
%	\bibliography{references}
 
\newcommand{\etalchar}[1]{$^{#1}$}
\def\cprime{$'$} \def\cprime{$'$}

\end{document}